\documentclass{article}[11pt]

\usepackage{amsmath}   
\usepackage{amssymb}
\usepackage{amsthm}
\usepackage{accents}
\usepackage{xcolor}
\usepackage{tikz-cd}

\newtheorem{Th}{Theorem}[section] 
\newtheorem{Prop}{Proposition}[section]   
\newtheorem{Lem}{Lemma}[section]   
\newtheorem{Coro}{Corollary}[section]   
  
\newtheorem{Rem}{Remark}[section]

\newcommand{\R}{\mathbb{R}}
\newcommand{\Z}{\mathbb{Z}}
\newcommand{\C}{\mathbb{C}}

\newcommand{\s}{{\rm S}}
\newcommand{\Sz}{\mathcal{S}}
\newcommand{\x}{\langle x\rangle}
\newcommand{\y}{\langle y\rangle}
\newcommand{\A}{{\mathcal A}}

\newcommand{\id}{\text{\rm id}}

\newcommand{\tr}{\text{\rm tr}\,}

\newcommand{\Div }{\mathop{\rm div}\nolimits}
\newcommand{\curl}{\mathop{\rm curl}\nolimits}

\newcommand{\dt}[1]{\accentset{\mbox{\bfseries .}}{#1}}

\newcommand{\df}{\accentset{\circ}}
\newcommand{\ddf}{\accentset{\,\,\,\circ}}
\newcommand{\dd}{{\rm d}}
\newcommand{\z}{{\overline{z}}}
\newcommand{\1}{\langle}
\newcommand{\2}{\rangle}

\newcommand{\floor}[1]{\lfloor #1 \rfloor}

\begin{document}

\title{Perfect fluid flows on $\R^d$ with growth/decay conditions at infinity}   
 
\author{R. McOwen and P. Topalov} 

\maketitle

\begin{abstract}  
We study the well-posedness and the spatial behavior at infinity of perfect fluid flows on $\R^d$ with initial data in a scale of 
weighted Sobolev spaces that allow spatial growth/decay at infinity as $|x|^\beta$ with $\beta<1/2$.
In particular, we show that the solution of the Euler equation generically develops an asymptotic expansion at infinity with 
non-vanishing asymptotic terms that depend analytically on time and the initial data.
We identify the evolution space for initial data in the Schwartz class with a certain space of symbols.
\end{abstract}   


\section{Introduction}\label{sec:introduction}
The motion of an incompressible perfect fluid in $\R^d$, $d\ge 2$, is described by the Euler equation
\begin{equation}\label{eq:euler}
\left\{
\begin{array}{l}
u_t+(u\cdot \nabla)\, u=-\nabla{\rm p},\quad \Div  u =0,\\
u|_{t=0}=u_0,
\end{array}
\right.
\end{equation}
where $u(t,x)$ is the velocity field and ${\rm p}(t,x)$ is the scalar pressure. 
In this paper, we consider the possibility of a spatial growth and decay of solutions of \eqref{eq:euler} in the context of a class of weighted 
Sobolev spaces that have been used by many authors. To define these, assume $1<p<\infty$, $\delta\in\R$, and  $m$ is a nonnegative 
integer. Let $W_\delta^{m,p}(\R^d)$ denote the Banach space obtained as the closure of $C_c^\infty(\R^d)$, i.e.\ the smooth 
functions with compact support, in the norm
\begin{equation}\label{def:W}
\|f\|_{W_\delta^{m,p}}=\sum_{|\alpha|\leq m} \|\x^{\delta+|\alpha|}\partial^\alpha f\|_{L^p}.
\end{equation}
Here $\alpha=(\alpha_1,\dots,\alpha_d)$ is a multi-index with $|\alpha|=\alpha_1+\cdots\alpha_d$, 
$\partial^\alpha$ denotes the partial derivative $\partial^{\alpha_1}_{x_1}\cdots\partial_{x_d}^{\alpha_d}$,
and $\x=\sqrt{1+|x|^2}$.
We shall write $W_\delta^{0,p}(\R^d)$ as $L_\delta^{p}(\R^d)$, i.e.\ a weighted $L^p$-space.
It was shown in \cite{Bartnik} and \cite{McOwenTopalov2} that $f\in W^{m,p}_\delta(\R^d)$ for $m>d/p$
implies  $f\in C^k(\R^d)$ for $0\leq k<m-d/p$ with
\begin{subequations}
\begin{equation}\label{eq:infinity_estimate}
\sup_{x\in\R^d}\x^{\delta+\frac{d}{p}+|\alpha|} |\partial^\alpha f(x)| \leq C\,\|f\|_{W^{m,p}_\delta}
\end{equation}
and in fact
\begin{equation}\label{eq:infinity_growth}
|x|^{\delta+\frac{d}{p}+|\alpha|}|\partial^\alpha f(x)|\to 0 \ \hbox{as}\ |x|\to\infty \ \hbox{for}\ |\alpha|<m-d/p.
\end{equation}
\end{subequations}
(Note also that $1/\x^\beta\in W^{m,p}_\delta$ for $m\ge 0$ and $\delta+d/p<\beta$.)
Of course, we are interested in vector fields, so we shall denote by $W^{m,p}_\delta=W_\delta^{m,p}(\R^d,\R^d)$ vector fields  
$u=(u_1,\dots,u_d)$ with components $u_k\in W^{m,p}_\delta(\R^d)$. In fact, henceforth
we shall use the notation $W^{m,p}_\delta$ regardless of whether we are considering scalar, vector, 
or even matrix-valued functions. 
For $m\ge 1$ we are particularly interested in 
divergence free vector fields, so we denote these by
\begin{equation}\label{circ-W}
\df W^{m,p}_{\delta}:=\big\{u\in W^{m,p}_{\delta}\,\big|\,\Div \, u=0\big\}.
\end{equation}
Note that $\df W^{m,p}_{\delta}$ is a closed subspace of $W^{m,p}_{\delta}$.
Our first result concerns weights $\delta\in\R$ in the range
\begin{equation}\label{delta-range}
-1/2<\delta+d/p <d+1.
\end{equation}
Using \eqref{eq:infinity_growth} we see that \eqref{delta-range} allows vector fields $u$ with mild growth, \ i.e.
$|u(x)|=O(|x|^\beta)$ as $|x|\to\infty$ for $0<\beta<1/2$, or decay, i.e.\ $|u(x)|=O(|x|^{\beta})$ for $\beta<0$.

\begin{Th}\label{th:main1}
Assume $m>3+d/p$, $1<p<\infty$, $d\ge 2$, and the weight $\delta$ satisfies \eqref{delta-range}.
Then, for any given $\rho>0$ there exists $\tau>0$ such that for any $u_0\in \df W^{m,p}_{\delta}$ 
with $\|u_0\|_{W^{m,p}_{\delta}}<\rho$ there exists a unique solution 
\begin{equation*}
u\in C\big([0,\tau],\df W^{m,p}_{\delta}\big)\cap 
C^1\big([0,\tau],\df W^{m-1,p}_{\delta}\big)
\end{equation*}
of the Euler equation \eqref{eq:euler} such that $\big|\nabla{\rm p}(t,x)\big|=o(1)$ as $|x|\to\infty$ for $t\in[0,\tau]$.
The solution depends continuously on the initial data $u_0\in\df W^{m,p}_\delta$. Moreover, for any fixed $t\in[0,\tau]$ 
the pressure ${\rm p}(t)$ is uniquely determined up to an additive constant.
\footnote{The continuity and the uniqueness are considered within 
the described class of solutions.}
\end{Th}

\noindent Theorem \ref{th:main1} generalizes the result of Cantor  \cite{Cantor2} which applies for $d\ge 3$, $p>d/(d-2)$, and 
$1+d/p<\delta+d/p<d-1$. (Cantor only requires $m>1+d/p$, but we make this same assumption in 
Theorem \ref{th:main1'} below for the case $-1/2<\delta+d/p<d-1$.)

\begin{Rem}\label{rem:below -1/2}
The bound $\delta+d/p<d+1$ in the parameter range \eqref{delta-range} is sharp.
Theorem \ref{th:main2} (b) and Proposition \ref{prop:asymptotic} stated below shows that Theorem \ref{th:main1} cannot 
be extended to $\delta+d/p\geq d+1$.
The bound $-1/2$ in \eqref{delta-range} appears since we require $\nabla{\rm p}(t,x)=o(1)$ as $x\to\infty$.
\end{Rem}

The proof of Theorem  \ref{th:main1} will be achieved in two steps: the range $-1/2 < \delta+d/p<d-1$ is proved in 
Section \ref{sec:ProofTheorem1} and the range $d-1\leq \delta+d/p<d+1$ is covered by case (a) in Theorem \ref{th:main2} below. 
Case (b) of Theorem \ref{th:main2} deals with the larger values of $\delta+d/p$ not covered in 
Theorem  \ref{th:main1}. Theorem \ref{th:main2}  is proved in Section \ref{sec:theorem_main2}; it relies on our analysis 
in \cite{McOwenTopalov3} of \eqref{eq:euler} on asymptotic spaces.
Take $\chi\in C^\infty(\R)$ such that $\chi(\rho)=0$ for $\rho\le 1$, $\chi(\rho)=1$ for $\rho\ge 2$, 
and $0\le\chi(\rho)\le 1$ for $1\le\rho\le 2$. We also set $r:=|x|$ and $\theta:=x/|x|$ for $x\ne 0$.
For $\rho>0$ denote  by $B_{\df W^{m,p}_\delta}(\rho)$ the  open ball of radius $\rho$ centered at the origin in $\df W^{m,p}_\delta$.

\begin{Th}\label{th:main2}
Assume that $m>3+d/p$, $1<p<\infty$, and $d\ge 2$. Then, we have:
\begin{itemize}
\item[\rm (a)] For any weight $\delta\in\R$ with $0<\delta+d/p<d+1$ and for any given $\rho>0$ there exists 
$\tau>0$ such that for any divergence free vector field $u_0\in B_{\df W^{m,p}_\delta}(\rho)$ there exists a unique solution of 
the Euler equation
\begin{equation*}
u\in C\big([0,\tau],\df W^{m,p}_{\delta}\big)\cap 
C^1\big([0,\tau],\df W^{m-1,p}_{\delta}\big).
\end{equation*}
\item[\rm (b)] For any weight $\delta\in\R$ with $\delta+d/p\geq d+1$ and for any given $\rho>0$ there exists 
$\tau>0$ such that for any divergence free vector field $u_0\in B_{\df W^{m,p}_\delta}(\rho)$ there exists a unique solution of 
the Euler equation of the form
\begin{equation}\label{eq:solution_general_form}
u(t)=\chi(r)\!\sum_{d+1\le k\le\delta+d/p}\!\frac{a_k(\theta,t)}{r^{k}}\,+\,f(t),\quad t\in[0,\tau],\footnote{In particular, 
by \eqref{eq:infinity_growth}, the solution $u$ has an asymptotic expansion at infinity of order equal to the integer part of 
$\delta+d/p$ with remainder in $W^{m,p}_\delta$.}
\end{equation}
where  $f(t)\in W^{m,p}_\delta$,
\begin{equation}\label{eq:f_case(b)}
f\in C\big([0,\tau],W^{m,p}_{\delta}\big)\cap C^1\big([0,\tau],W^{m-1,p}_{\delta}\big),
\end{equation}
\begin{equation}\label{eq:a_case(b)}
a_k\in C^1\left([0,\tau],C\big(S^{d-1},\R^{d}\big)\right),
\quad d+1\le k\le\delta+d/p,
\end{equation}
and for any given $t\in[0, \tau]$, the components of $a_k(\theta,t)$ are eigenfunctions of the Laplace operator $-\Delta_S$ on 
the unit sphere $S^{d-1}$ with eigenvalue $\lambda_{k-d+2}=k(k-d+2)$.
\end{itemize}
The solution in {\rm (a)} and {\rm (b)} depends continuously on the initial data $u_0\in\df W^{m,p}_{\delta}$.
It is {global} in time for $d=2$. 
Moreover, the asymptotic coefficients $a_k$, $d+1\le k\le\delta+d/p$, in \eqref{eq:solution_general_form}
are {analytic} as functions of time and the initial data, i.e.\ as maps 
$a_k : [0,\tau]\times B_{\df W^{m,p}_\delta}(\rho)\to C\big(S^{d-1},\R^d\big)$.
\end{Th}

\begin{Rem} \label{rem:extended_well-posedness}
The continuous dependence on the initial data in Theorem \ref{th:main2} {\rm (b)} means that the 
data-to-solution map $u_0\mapsto\big(a_{d+1},...,a_N;f\big)$,
\begin{equation}\label{eq:extended_well-posedness}
\quad B_{\df W^{m,p}_\delta}(\rho)\to 
C\big([0,\tau],C\big(S^{d-1},\R^{d}\big)^{N-d}\times W^{m,p}_{\delta}\big)\cap 
C^1\big([0,\tau],C\big(S^{d-1},\R^{d}\big)^{N-d}\times W^{m-1,p}_{\delta}\big),
\end{equation}
is continuous, where $N$ is the integer part of $\delta+d/p$. In this way, for $\delta+d/p\ge d+1$ the Euler equation is well-posed 
in the sense described in \eqref{eq:extended_well-posedness}.  
Altogether, with this extended notion of well-posedness, the Euler equation is well-posed in the weighted Sobolev space $
W^{m,p}_\delta$ for any $\delta+d/p>-1/2$.
\end{Rem}


\begin{Rem}\label{rem:analyticity}
The analyticity in time established e.g.\ in \cite{Serfati1} concerns the solution of the Euler equation in Lagrangian 
coordinates. In contrast, the analyticity of the coefficients $a_k$ for $d+1\le k<\delta+d/p$ in Theorem \ref{th:main2} concerns
the solution \eqref{eq:solution_general_form} of the Euler equation \eqref{eq:euler} written in Eulerian coordinates. 
Note that the remainder $f$ in \eqref{eq:solution_general_form} is {not} necessarily analytic in time.
\end{Rem}

\begin{Rem}
Theorem \ref{th:main2} generalizes Corollary 1.1 and Corollary 1.2 in \cite{SultanTopalov} which did not cover 
integer values of $\delta+2/p$ in the case when $d=2$.
\end{Rem}



Let us now discuss the asymptotic terms appearing in the solution \eqref{eq:solution_general_form} given by
Theorem \ref{th:main2} {\rm (b)}. The following proposition implies that for generic initial data $u_0$ in $\df W^{m,p}_\delta$  with 
$\delta+d/p\ge d+1$, {\em all} asymptotic terms in \eqref{eq:solution_general_form} do appear and do {\em not} vanish generically.
More specifically, we have

\begin{Prop}\label{prop:asymptotic} 
Assume that $m>3+d/p$, $1<p<\infty$, and $d\ge 2$. Then, under the assumptions of Theorem \ref{th:main2} {\rm (b)}, 
there exists an open {dense} set $\mathcal{N}$ in $\df W^{m,p}_\delta$ such that for any initial data $u_0\in\mathcal{N}$ 
and for any $d+1\le k<\delta+d/p$ and $1\le j\le d$ the $j$-th component $a_k^j(t)$ of the asymptotic coefficient 
$a_k(t)\in C(S^{d-1},\R^d)$ of the solution \eqref{eq:solution_general_form} does {\em not} vanish in $C\big(S^{d-1},\R\big)$ for 
all but finitely many $t\in[0,\tau]$.
\end{Prop}

Our next result shows that the interval of existence $[0,\tau]$, $\tau>0$, in Theorem \ref{th:main2} can be chosen independent of
the regularity exponent $m>3+d/p$ and the weight $\delta+d/p>0$. More specifically, we have the following ``no gain no loss'' result.

\begin{Prop}\label{prop:no_gain_no_loss}
Take a regularity exponent $m_0>3+d/p$, $1<p<\infty$, a weight $\delta_0+d/p>0$, a radius $\rho>0$, and let $[0,\tau]$, 
$\tau>0$, be the interval of existence of the solution \eqref{eq:solution_general_form} of the Euler equation \eqref{eq:euler} with 
initial data $u_0\in B_{\df W^{m_0,p}_{\delta_0}}(\rho)$, given by Theorem \ref{th:main2}. 
Then, for any $m\ge m_0$, $\delta\ge\delta_0$, and for any initial data 
$u_0\in B_{\df W^{m_0,p}_{\delta_0}}(\rho)\cap W^{m,p}_\delta$ there exists a unique solution \eqref{eq:solution_general_form} of 
the Euler equation \eqref{eq:euler} that is defined on $[0,\tau]$, satisfies \eqref{eq:f_case(b)}, \eqref{eq:a_case(b)}, and depends continuously 
on the initial data (see \eqref{eq:extended_well-posedness}).\footnote{Since the asymptotic sum in \eqref{eq:solution_general_form} vanishes 
for $\delta+d/p<d+1$, formula \eqref{eq:solution_general_form} continues to hold also in the case of Theorem \ref{th:main1} and 
Theorem \ref{th:main2} {\rm (a)}.}
\end{Prop}

\medskip

\noindent{\em Solutions in spaces of symbols.}
Proposition \ref{prop:no_gain_no_loss} implies that the Euler equation \eqref{eq:euler}, 
with initial data in the Schwartz space $\Sz=\Sz(\R^d,\R^d)$, is locally well-posed in the space $\mathcal{I}^\infty$ of $C^\infty$ 
vector fields on $\R^d$ that allow an {\em infinite asymptotic expansion}
\begin{equation}\label{eq:symbol_asymptotic}
u(x)\sim\sum_{k\ge d+1}\frac{a_k(\theta)}{r^{k}}\quad\text{\rm as}\quad|x|\to\infty
\end{equation}
with coefficients $a_k(\theta)$, $k\ge d+1$, as in Theorem \ref{th:main2} {\rm (b)}.
The asymptotic formula \eqref{eq:symbol_asymptotic} means that for any $N\ge d$ and for any multi-index $\alpha\in\Z_{\ge 0}^d$
there exists a constant $C_{N,\alpha}>0$ such that for any $x\in\R^d$,
\begin{equation}\label{eq:symbols_estimates}
\Big|\partial^\alpha\Big(u(x)-\chi(r)\sum_{d+1\le k\le N}\frac{a_k(\theta)}{r^{k}}\Big)\Big|\le\frac{C_{N,\alpha}}{\x^{N+1+|\alpha|}}.
\end{equation}
Hence, we can think of $\mathcal{I}^\infty$ as a {\em space of symbols} (cf.\ \cite{BS,KPST}).
The best constants $C_{N,\alpha}$ for $N\ge d$ and $|\alpha|\le N$ in \eqref{eq:symbols_estimates} equip the space of symbols 
$\mathcal{I}^\infty$ with a countable set of semi-norms that  induce a Fr\'echet topology on $\mathcal{I}^\infty$. 
Denote by $\df\Sz$ the space of divergence free vector fields in $\Sz$. 
We have the following

\begin{Th}\label{th:symbol_classes} 
For any $u_0\in\df\Sz$ there exist $\tau>0$ and a unique solution of the Euler equation $u\in C^1\big([0,\tau],\mathcal{I}^\infty\big)$
that depends {continuously} on the initial data. The asymptotic coefficients $a_k(t)\in C\big(S^{d-1},\R^d\big)$ for $k\ge d+1$ depend 
{analytically} on $t\in[0,\tau]$ and the initial data.
\end{Th}

\noindent The theorem follows directly from Theorem \ref{th:main2} and Proposition \ref{prop:no_gain_no_loss}. 
In Section \ref{sec:theorem_main2} we show that, for generic initial data in $\df\Sz$, {\em any} asymptotic term in the 
infinite asymptotic expansion \eqref{eq:symbol_asymptotic} of the solution  $u\in C^1\big([0,\tau],\mathcal{I}^\infty\big)$ 
does appear. More specifically we have

\begin{Coro}\label{coro:schwartz_data} 
There exists a dense set $\mathcal{N}$ in $\df\Sz$ such that for any initial data $u_0\in\mathcal{N}$ and for any $k\ge d+1$ and
$1\le j\le d$ the $j$-th component $a_k^j(t)$ of the asymptotic coefficient $a_k(t)\in C(S^{d-1},\R^d)$ in 
the asymptotic expansion \eqref{eq:symbol_asymptotic} of the solution $u\in C^1\big([0,\tau],\mathcal{I}^\infty\big)$
given by Theorem \ref{th:symbol_classes} does {not} vanish in $C(S^{d-1},\R)$ for {all} but finitely many $t\in[0,\tau]$.
\end{Coro}

\noindent In this sense, the space of symbols $\mathcal{I}^\infty$ is the {\em minimal evolution space} for the Euler equation with
initial data in the Schwartz space. 

\medskip

\noindent{\em Related work.} Now let us describe previous results and how our work relates to them. First we list works that consider  
\eqref{eq:euler}  but only for $d=2$. 
Solutions of the Euler equation for (not necessarily bounded) domains in $\R^2$ were constructed  by Wolibner in \cite{Wol}.
Kato and Ponce \cite{KatoPonce1}, \cite{KatoPonce2}, showed the well-posedness of \eqref{eq:euler} as the zero-viscosity limit of the 
Navier-Stokes equation in the standard $L^p$-Sobolev spaces $H^{m,p}(\R^2)$ for $m>1+d/p$; this requires the vector fields to vanish at infinity. 
To allow the velocity field to be merely bounded at infinity, several authors also considered conditions on the vorticity 
$\omega={\rm curl}\,u$. For example, Serfati \cite{Serfati2} showed that for $u_0,\omega_0\in L^\infty$, there is a unique solution 
$u\in L^\infty$ of \eqref{eq:euler} with the pressure {\rm p} determined up to an additive constant and satisfying ${\rm p}=o(|x|)$ as 
$|x|\to\infty$; under the same conditions,   Kelliher \cite{Kelliher} gave a characterization of the behavior of the solution $u$ as $|x|\to\infty$. 
Under a decay condition on the vorticity,  Benedetto, Marchioro, and Pulvirenti \cite{BMP} allowed a sublinear growth of the solution, 
namely if $|u_0(x)|= O\big(|x|^\beta\big)$ as $|x|\to \infty$ where $\beta<1$ and $\omega_0\in L^p\cap L^\infty(\R^2)$ for $p<2/\beta$, 
then the unique solution $u$ retains these properties as $|x|\to\infty$. Very recently, Elgindi and Jeong \cite{EJ}  allowed up to linear growth for 
$u$, namely $|u(x)|=O(|x|)$ as $|x|\to\infty$, but instead of requiring $\omega$ to decay at infinity, they imposed symmetry conditions. 
Without any decay or symmetry conditions on $\omega\in L^\infty(\R^2)$, Cozzi and Kelliher \cite{CK} allowed growth 
$|u(x)|=O\big(|x|^\beta\big)$ as $|x|\to\infty$ for $\beta<1/2$, which is similar to the growth allowed in our Theorem \ref{th:main1} 
(although our initial conditions and solutions require $|\omega(x)|=O\big(|x|^{\beta-1}\big)$ as $|x|\to\infty$).  Most of the results considered 
above do not discuss the well-posedness of the obtained solutions. 

Now we describe works that consider \eqref{eq:euler} for $d\ge 2$. Kato \cite{Kato2} showed  the well-posedness of \eqref{eq:euler} as the 
zero-viscosity limit of the Navier-Stokes equation in the standard $L^2$-Sobolev spaces $H^{m}(\R^3)$ for $m\ge 3$. 
 Following the approach of  Ebin and Marsden \cite{EM}, Cantor \cite{Cantor2} showed the well-posedness of 
\eqref{eq:euler} in the weighted Sobolev spaces $W^{m,p}_\delta$, but only for $d\ge 3$, $p>d/(d-2)$, and $1+d/p< \delta+d/p<d-1$.
The fact that the solution of \eqref{eq:euler} decays at the rate $O\big(1/r^{d+1}\big)$ when the initial data has compact support or 
is rapidly decaying at infinity  was first noticed in \cite{DobrShaf} for $d=3$; they do not obtain a full asymptotic expansion and do not study
dependence on initial data (cf.\ also \cite{BM,KR} for related results on the Navier-Stokes equation). 
The well-posedness results of \cite{Kato2} and \cite{Cantor2} all involve  velocity fields that vanish at infinity. 
A special class of non-decaying solutions of \eqref{eq:euler} were obtained in \cite{McOwenTopalov3}, in which we showed well-posedness 
in the {\em asymptotic spaces} $\A^{m,p}_{N;0}$ that we define in Appendix \ref{sec:C} and use in Section \ref{sec:theorem_main2} of this paper. 
The results obtained in \cite{McOwenTopalov3} apply for $d\ge 2$; but very recently, for $d=2$, it was shown in \cite{SultanTopalov} 
that \eqref{eq:euler} is globally well-posed in asymptotic spaces analogous to those considered in \cite{McOwenTopalov3} but with 
asymptotic expansions which do not involve logarithmic terms. We shall use results from both \cite{McOwenTopalov3} and 
\cite{SultanTopalov} below. However, to our best knowledge,  Theorem \ref{th:main1} is the only result that allows classical solutions 
of \eqref{eq:euler} to grow as $|x|\to\infty$ in dimension {\em greater} than two.

Finally, let us discuss previous work on the existence of solutions of partial differential equations in spaces of symbols
(cf. Theorem \ref{th:symbol_classes}). In the one dimensional case, spaces of symbols were considered by Bondareva and Shubin in 
\cite{BS} where they proved that the Korteweg-de Vries equation (KdV) allows a unique solution in such spaces.
A similar result was proved in \cite{KPST} for the modified KdV equation. Note that the space of symbols $\mathcal{I}^\infty$ appears 
naturally in the case of the Euler equation due to the non-local nature of its right hand side. 


\section{Spaces of maps on $\R^d$}\label{sec:Diffeos}
We will use groups of diffeomorphisms on $\R^d$ whose behavior at infinity is modeled on the weighted Sobolev spaces 
$W^{m,p}_\delta$.  This was done in \cite{Cantor1} for $\delta\ge 0$, and in \cite{McOwenTopalov2} for $\delta+d/p>0$ in 
the context of asymptotic spaces. However, to include $-1/2<\delta+d/p<0$, we need to generalize these previous results.
For $m>1+d/p$ and $\gamma+d/p>-1$ consider the set of maps $\R^d\to\R^d$,
\begin{equation}\label{def:D}
{\mathcal D}^{m,p}_\gamma:=\big\{\varphi : \R^d\to\R^d\,\big|\,\varphi=\id+w,\,w\in W^{m,p}_\gamma\,
\text{and}\,\det(\dd\varphi)>0\big\}.
\end{equation}
(Throughout this paper we denote the identity map on $\R^d$ by $\id$ and the identity matrix by ${\rm I}$.)
For $\varphi=\id+w\in {\mathcal D}^{m,p}_\gamma$, the above restrictions on $m$ and $\gamma$ imply by \eqref{eq:infinity_estimate} 
that $|w(x)|=O(|x|^{1-\mu})$ as $|x|\to\infty$ for some $\mu>0$.  Moreover, we have $\dd\varphi={\rm I}+\dd w$ with 
$|\dd w(x)|=O(|x|^{-\mu})$ as $|x|\to\infty$.  This and the fact that $m>1+d/p$ imply that there exists $\varepsilon>0$ such that 
$\det(\dd\varphi)>\varepsilon>0$. By the Hadamard-Levi's theorem, one then sees that ${\mathcal D}^{m,p}_\gamma$ consists of orientation 
preserving $C^1$-diffeomorphisms of $\R^d$ (cf. Corollary \ref{co:D-open} in Appendix \ref{sec:A}). 
Note that ${\mathcal D}^{m,p}_\gamma$ can be identified with an open set in $W^{m,p}_\gamma$. In this way, 
${\mathcal D}^{m,p}_\gamma$ is a Banach manifold modeled on $W^{m,p}_\gamma$.
Recall from \cite[Proposition 2.2, Lemma 2.2]{McOwenTopalov2} that, for $m>d/p$ and any $\delta_1,\delta_2\in\R$,  pointwise multiplication of 
functions $(f,g)\mapsto fg$ defines a continuous map
\begin{equation}\label{eq:W-multiplication}
W^{m,p}_{\delta_1} \times W^{m,p}_{\delta_2} \to W^{m,p}_{\delta_1+\delta_2+\frac{d}{p}}
\end{equation}
and for any $m\ge 0$, $\delta\in\R$, $1\le j\le d$,
\begin{equation}\label{eq:W-derivative}
\partial_{x_j} : W^{m+1,p}_\delta\to W^{m,p}_{\delta+1}
\end{equation}
is bounded. Now we want to show that ${\mathcal D}^{m,p}_\gamma$ is a topological group under composition of maps. 
Since for $\varphi,\psi\in{\mathcal D}^{m,p}_\gamma$, $\varphi=\id+w$ we have that $\varphi\circ\psi=\psi+w\circ\psi$, 
we see that the continuity of $(\varphi,\psi)\mapsto \varphi\circ\psi$ is equivalent to the continuity of the map
$(w,\psi)\mapsto w\circ\psi$, $W^{m,p}_\gamma\times{\mathcal D}^{m,p}_\gamma  \to W^{m,p}_\gamma$. 
However, it will be important to also consider $w\circ\varphi$ for $w\in W^{m,p}_\delta$ with weights $\delta$ 
different from $\gamma$. We have the following

\begin{Th}\label{th:group} 
Assume that $m>1+d /p$, $\gamma+d/p>-1$ and $\delta\in\R$.
\begin{itemize}
\item[\rm (a)] The composition $(w,\varphi)\mapsto w\circ\varphi$ is continuous as a map
$W^{m,p}_\delta\times{\mathcal D}^{m,p}_\gamma  \to W^{m,p}_\delta$
and $C^1$ as a map $W^{m+1,p}_\delta\times{\mathcal D}^{m,p}_\gamma \to W^{m,p}_\delta$.
\item[\rm (b)] The inverse $\varphi\mapsto \varphi^{-1}$ is continuous as a map 
${\mathcal D}^{m,p}_\gamma\to {\mathcal D}^{m,p}_\gamma$ and $C^1$ as a map 
${\mathcal D}^{m+1,p}_\gamma\to {\mathcal D}^{m,p}_\gamma$.
\end{itemize}
\end{Th}
\noindent
Theorem \ref{th:group} and its proof are analgous to Propositions 5.1 and 5.2 in \cite{McOwenTopalov2} (cf. also the  proof
of \cite[Theorem 2.1]{SultanTopalov}); for the details we refer to Appendix \ref{sec:A}.

For fixed $\varphi\in {\mathcal D}^{m,p}_\gamma$, we can operate on functions and vector fields by {\it right translation}:
\begin{equation}\label{def:R}
R_\varphi(v):=v\circ \varphi.
\end{equation}
In fact, under the hypotheses of Theorem \ref{th:group}, we see that $R_\varphi : W^{m,p}_\delta\to W^{m,p}_\delta$ is linear, 
bounded, and has a bounded inverse, $R_{\varphi^{-1}}$. Hence, we have

\begin{Coro}\label{co:R_phi-isom}  
Assume $m\ge m_0>1+d/p$ and  $\gamma+d/p>-1$. For any $\varphi\in {\mathcal D}^{m,p}_\gamma$ and $\delta\in\R$, 
the map $R_\varphi : W^{m_0,p}_\delta\to W^{m_0,p}_\delta$ is a linear isomorphism.
\end{Coro}

We can also conjugate differential operators with $R_\varphi$ and its inverse. For example, suppose 
$\varphi\in{\mathcal D}^{m,p}_\gamma$, $f\in W^{m_0,p}_\delta$ is scalar-valued, and $v\in W^{m_0,p}_\delta$ is a vector field. 
Consider the maps
\begin{subequations}
\begin{equation}\label{eq:conjugate-nabla}
(\varphi,f)\mapsto R_\varphi\circ \nabla \circ R_{\varphi^{-1}}(f),
\end{equation}
\begin{equation}\label{eq:conjugate-div}
(\varphi,v)\mapsto R_\varphi\circ {\rm div} \circ R_{\varphi^{-1}}(v),
\end{equation}
\begin{equation}\label{eq:conjugate-Delta}
(\varphi,f)\mapsto \Delta_\varphi(f):=R_\varphi\circ \Delta \circ R_{\varphi^{-1}}(f).
\end{equation}
\end{subequations}
These maps are not just continuous in $\varphi$ and $f$ (or $v$) as asserted in Theorem \ref{th:group}, 
but their special structure makes them smooth\footnote{By ``smooth'' we mean $C^\infty$-smooth. However,
these maps will actually be shown to be real analytic.}.

\begin{Lem}\label{le:nabla_phi} 
Assume $\gamma+d/p>-1$ and  $\delta\in\R$.
\begin{itemize}
\item[\rm (a)] If $m\ge m_0>1+d/p$ then \eqref{eq:conjugate-nabla} is smooth as a map
${\mathcal D}^{m,p}_\gamma \times W^{m_0,p}_\delta(\R^d) \to W^{m_0-1,p}_{\delta+1}$
and \eqref{eq:conjugate-div} is smooth as a map 
${\mathcal D}^{m,p}_\gamma\times W^{m_0,p}_\delta\to W^{m_0-1,p}_{\delta+1}$
\item[\rm (b)] If $m\ge m_0>2+d/p$ then \eqref{eq:conjugate-Delta} is smooth as a map
${\mathcal D}^{m,p}_\gamma \times W^{m_0,p}_\delta \to W^{m_0-2,p}_{\delta+2}$.
\end{itemize}
\end{Lem}

\begin{proof}[Proof of Lemma \ref{le:nabla_phi}]
In (a) we have $\varphi,f\in C^1$, so we can use the chain rule $\dd (f\circ\varphi)=
(\dd f\circ\varphi)\cdot\dd\varphi$ and the relationship $\dd(\varphi^{-1})=[(\dd\varphi)\circ\varphi^{-1}]^{-1}$, i.e.\ 
the inverse of the Jacobian matrix, to conclude
\[
R_\varphi\circ \dd  \circ R_{\varphi^{-1}} (f) = \dd f\cdot(\dd \varphi)^{-1}.
\]
If we write $\varphi=\id+w$, with $w\in W^{m,p}_\gamma$, then $\dd \varphi={\rm I}+\dd w$ where  the matrix 
$\dd w$ has components in $W^{m-1,p}_{\gamma+1}$. But, using the adjoint formula for the inverse of a matrix, 
we see that $(\dd \varphi)^{-1}$ is a matrix with components which are obtained by taking sums and products of elements of 
$\dd \varphi$, and a product by $1/\det(\dd\varphi)$. Then, one sees from 
\eqref{eq:W-multiplication}  (see e.g. the proof of \cite[Lemma 2.3]{SultanTopalov}) that $(\dd \varphi)^{-1}={\rm I}+T(w)$
where $w\mapsto T(w)$, $W^{m,p}_\gamma\to W^{m-1,p}_{\gamma+1}$, is real-analytic.
Moreover, since $\gamma+1+d/p>0$, the multiplication by ${\rm I}+T(w) $ defines a real analytic 
map $W^{m_0-1,p}_{\delta+1}\to W^{m_0-1,p}_{\delta+1}$.
In summary, we see from \eqref{eq:W-multiplication} and \eqref{eq:W-derivative} that the mapping 
$(\varphi,f)\mapsto \dd  f \cdot(\dd \varphi)^{-1}$, ${\mathcal D}^{m,p}_\gamma\times W^{m_0,p}_\delta \to W^{m_0-1,p}_{\delta+1}$, 
is real-analytic. This confirms that \eqref{eq:conjugate-nabla} is real analytic and similar calculations show 
\eqref{eq:conjugate-div} is real analytic.  By rewriting \eqref{eq:conjugate-Delta} as
\[
(\varphi,f)\mapsto\big(R_\varphi\circ\Div\circ R_{\varphi^{-1}}\big)\big(R_\varphi\circ \nabla \circ R_{\varphi^{-1}}\big)(f),
\]
where $(\varphi,f)\in {\mathcal D}^{m,p}_\gamma \times W^{m_0,p}_\delta$, we see that (b) follows from (a).
\end{proof}

The reason to introduce the group of diffeomorphisms is to reduce the partial differential equation in $W^{m,p}_\delta$ to an 
ordinary differential equation in ${\mathcal D}^{m,p}_\delta\times W^{m,p}_\delta$. 
The following will be used in Section \ref{sec:ProofTheorem1} as part of that analysis.

\begin{Prop}\label{pr:flowODE} 
For $m>1+d/p$ and $\gamma+d/p>-1$, assume $u\in C\big([0,\tau],W_\gamma^{m,p}\big)$.
Then there is a unique solution $\varphi\in C^1\big([0,\tau],{\mathcal D}^{m,p}_\gamma\big)$ of the equation
\begin{equation}\label{eq:flowODE}
\left\{
\begin{array}{l}
\dot\varphi=u\circ\varphi,\\
\varphi|_{t=0}=\id.
\end{array}
\right.
\end{equation}
\end{Prop}

\noindent This Proposition can be proved exactly like Proposition 2.1 in \cite{McOwenTopalov3} (cf. \cite{EM} for the case
of compact manifolds). We simply observe here that the condition 
$m>1+d/p$ is used with Theorem \ref{th:group}(a) to view $u\circ\varphi$ as a $C^1$ vector field on 
${\mathcal D}^{m-1,p}_\gamma$ and obtain a solution $\varphi \in C^1\big([0,\tau],{\mathcal D}^{m-1,p}_\gamma\big)$; 
a bootstrap argument is then used to show $\varphi\in C^1\big([0,\tau],{\mathcal D}^{m,p}_\gamma\big)$.

\section{Mapping properties of the Laplacian and its inverse}\label{sec:Laplacian}
The mapping on scalar functions
\begin{equation}\label{Lap_on_W}
\Delta: W^{m,p}_\delta\to W^{m-2,p}_{\delta+2}
\end{equation}
for $m\ge 2$ was studied in \cite{McOwen} and found to have the following properties:
\begin{itemize}
\item[$(A)$] it is an isomorphism for $0<\delta+\frac{d}{p}<d-2$ when $d\geq 3$,
\item[$(B)$] it is surjective for $-1<\delta+\frac{d}{p}<0$ with the space of constant functions $\mathcal{N}_0$ as a nullspace,
\item[$(C)$] it is injective with cokernel ${\mathcal N}_0$ for $d-2<\delta+d/p<d-1$. Hence, if $g\in W^{m-1,p}_{\delta+2}$ 
satisfies $\int_{\R^d}g\,dx=0$, then $\Delta^{-1}g\in W^{m,p}_{\delta}$.
\end{itemize}
In fact, it was also shown in \cite{McOwen} that for $-k-1<\delta+d/p<-k$ where $k=0,1,2,\dots$, the nullspace of 
\eqref{Lap_on_W} consists of harmonic polynomials of degree $k$ and these appear as the cokernel of \eqref{Lap_on_W} for 
$d-2+k<\delta+d/p<d-1+k$. In \cite{McOwenTopalov2} it was shown that the presence of cokernel for $\delta+d/p>d-2$ leads 
to asymptotics when inverting $ \Delta$ on $W^{m-2,p}_{\delta+2}$  (cf. Proposition \ref{prop:inverting_the_laplace_operator} 
in Appendix \ref{sec:B} for an extended version of this result); this was the genesis of the {\it asymptotic spaces} 
$\A^{m,p}_{n,N}$ studied there. These asymptotic spaces will be used in a subsequent section of this paper;  
but for now, let us explore further consequences of $(A)$ and $(B)$.

In view of $(B)$, for $m\ge 1$ the map $\Delta : W^{m+1,p}_{\delta-1}\to W^{m-1,p}_{\delta+1}$ has nontrivial nullspace 
${\mathcal N}_0$ for $0<\delta+d/p<1$, so we can form the quotient space, $W^{m+1,p}_{\delta-1}/{\mathcal N}_0$.
This is a Banach space under the usual quotient norm, and we can reformulate $(B)$ as the statement:
\begin{equation}\label{eq:D-isomorphism1}
\Delta : W^{m+1,p}_{\delta-1}/{\mathcal N}_0\to W^{m-1,p}_{\delta+1}\ \hbox{is an isomorphism for $0<\delta+\frac{d}{p}<1$}.
\end{equation}
Note that $\nabla : W^{m+1,p}_{\delta-1}/{\mathcal N}_0\to W^{m,p}_\delta$ is well-defined and bounded, so 
$\nabla\circ\Delta^{-1} : W^{m-1,p}_{\delta+1}\to W^{m,p}_\delta$ is bounded. Combining this with the fact (see case $(A)$) that 
\begin{equation}\label{eq:D-isomorphism2}
\Delta : W^{m+1,p}_{\delta-1}\to W^{m-1,p}_{\delta+1}\ \hbox{is an isomorphism for $1<\delta+\frac{d}{p}<d-1$}
\end{equation}
we conclude that
$\nabla\circ\Delta^{-1} : W^{m-1,p}_{\delta+1}\to W^{m,p}_\delta$
is bounded for 
$0<\delta+\frac{d}{p}<d-1$ but $\delta+\frac{d}{p}\not=1$.
Hence, we have 
 


\begin{Lem}\label{le:ND}
For $m\ge 1$, $0<\delta+\frac{d}{p}<d-1$, and $\delta+\frac{d}{p}\ne 1$, the operator
\begin{equation}\label{eq:ND}
\nabla\circ\Delta^{-1} : W^{m-1,p}_{\delta+1}\to W^{m,p}_\delta
\end{equation}
is bounded.
\end{Lem}

Assume that $d\ge 2$, $m>1+d/p$, and $\gamma+d/p>-1$. For $0<\delta+d/p<d-1$ and $\delta+d/p\ne 1$ consider
the map
\begin{equation}\label{eq:ND-conjugate}
(\varphi,f)\mapsto R_\varphi\circ \nabla\circ\Delta^{-1}\circ R_{\varphi^{-1}}(f),\quad 
{\mathcal D}^{m,p}_\gamma\times W^{m-1,p}_{\delta+1}\to  W^{m,p}_\delta,
\end{equation}
which is well defined by Corollary \ref{co:R_phi-isom} and Lemma \ref{le:ND}. 
We have

\begin{Prop}\label{pr:ND-conjugate}
For $0<\delta+d/p<d-1$ and $\delta+d/p\ne 1$ the map \eqref{eq:ND-conjugate} is well defined and smooth.
\end{Prop}

\begin{proof}[Proof of Proposition \ref{pr:ND-conjugate}]
In view of Corollary \ref{co:R_phi-isom}, \eqref{eq:D-isomorphism1}, and \eqref{eq:D-isomorphism2}, we can write
\begin{equation}\label{eq:ND-factorization}
R_\varphi\circ\nabla\circ\Delta^{-1}\circ R_{\varphi^{-1}}(f)= 
\big(R_\varphi\circ\nabla\circ R_{\varphi^{-1}}\big)\circ\big(R_\varphi\circ\Delta^{-1}\circ R_{\varphi^{-1}}\big)(f)
\end{equation}
where 
\begin{equation}\label{eq:D-component}
(\varphi,f)\mapsto R_\varphi\circ\Delta^{-1}\circ R_{\varphi^{-1}}(f),\quad
{\mathcal D}^{m,p}_\gamma\times W^{m-1,p}_{\delta+1}\to 
\left\{
\begin{array}{l}
W^{m+1,p}_{\delta-1}/{\mathcal N}_0\quad\text{for}\quad 0<\delta+d/p<1,\\
W^{m+1,p}_{\delta-1}\quad\text{for}\quad 1<\delta+d/p<d-1,
\end{array}
\right.
\end{equation}
and
\begin{equation}\label{eq:N-component}
(\varphi,g)\mapsto R_\varphi\circ\nabla\circ R_{\varphi^{-1}}(g),\quad
{\mathcal D}^{m,p}_\gamma\times W^{m+1,p}_{\delta-1}\to W^{m,p}_\delta.
\end{equation}
It follows from Lemma \ref{le:nabla_phi}(a) that the map \eqref{eq:N-component} is smooth.
We will prove the smoothness of the map \eqref{eq:D-component} by following the proof of Proposition 5.1 in 
\cite{McOwenTopalov3}. Consider the spaces $F:=W^{m-1,p}_{\delta+1}$ and
\[
E:=\left\{
\begin{array}{l}
W^{m+1,p}_{\delta-1}/{\mathcal N}_0\quad\text{for}\quad 0<\delta+d/p<1,\\
W^{m+1,p}_{\delta-1}\quad\text{for}\quad 1<\delta+d/p<d-1.
\end{array}
\right.
\]
Let $GL(E,F)$ be the group of linear isomorphisms $G : E\to F$ considered as a subspace of ${\mathcal L}(E,F)$, 
the Banach space of all bounded linear maps $E\to F$. By using Neumann series, one sees that $GL(E,F)$ is an open 
set in ${\mathcal L}(E,F)$ and the map 
\begin{equation}\label{eq:GL-inverse}
G\mapsto G^{-1},\quad GL(E,F)\to GL(F,E),
\end{equation}
is real analytic. It follows from \eqref{eq:D-isomorphism1}, \eqref{eq:D-isomorphism2}, and Corollary \ref{co:R_phi-isom} 
that for any $\varphi\in{\mathcal D}^{m,p}_\gamma$,
\[
\Delta_\varphi\equiv R_\varphi\circ\Delta\circ R_{\varphi^{-1}}\in GL(E,F)\quad\text{and}\quad
\Delta_\varphi^{-1}=R_\varphi\circ\Delta^{-1}\circ R_{\varphi^{-1}}\in GL(F,E).
\]
By Lemma \ref{le:nabla_phi}(b), the map \eqref{eq:conjugate-Delta} is real analytic. Since this map is linear in its
second argument we conclude that the map
\[
\varphi\mapsto\Delta_\varphi\equiv R_\varphi\circ\Delta\circ R_{\varphi^{-1}},\quad
{\mathcal D}^{m,p}_\gamma\to GL(E,F),
\]
is real analytic. Composing this with the real analytic map \eqref{eq:GL-inverse} we find that 
\[
\varphi\mapsto\Delta_\varphi^{-1},\quad{\mathcal D}^{m,p}_\gamma\to GL(F,E),
\] 
is real analytic. Hence the map \eqref{eq:D-component} is smooth.
Finally, the smoothness of the map \eqref{eq:ND-conjugate} follows from the factorization \eqref{eq:ND-factorization}
and the fact that, by Lemma \ref{le:nabla_phi}{\rm (a)}, the map
\[
(\varphi,g)\mapsto R_\varphi\circ\nabla\circ R_{\varphi^{-1}}(g),\quad{\mathcal D}^{m,p}_\gamma\times E\to F,
\]
is well defined and smooth.
\end{proof}

\section{The case $-1/2<\delta+d/p<d-1$}\label{sec:ProofTheorem1} 
Let $u\in C\big([0,\tau],\df W^{m,p}_\delta\big)\cap C^1\big([0,\tau],\df W^{m-1,p}_\delta\big)$, $m>1+d/p$, $\delta\in\R$, 
be a solution of the Euler equation \eqref{eq:euler}.
Applying $\Div$ to both sides of $u_t+(u\cdot\nabla u)\,u=-\nabla p$ and using $\Div u=0$, we obtain 
\begin{equation}\label{Lap-p}
\Delta {\rm p}=-Q(u),
\end{equation}
where
\begin{equation}\label{def:Q}
Q(u):= \tr\big(\dd u\big)^2=\Div \big(u\cdot\nabla\,u\big);
\end{equation}
here $(\dd u)^2$ denotes the square of the Jacobian matrix and $\tr$ denotes the trace of a matrix.
Note that $Q$ maps vector functions to scalar functions;
in fact, we have the following mapping property.

\begin{Lem}\label{le:Q-map} 
If $m>1+d/p$, then $u\mapsto Q(u)$ defines a smooth map
\begin{equation}\label{Q-map}
Q : W^{m,p}_\delta \to W^{m-1,p}_{2\delta+2+\frac{d}{p}} \quad\hbox{for any $\delta\in\R$}.
\end{equation}
\end{Lem}

\begin{proof}[Proof of Lemma \ref{le:Q-map}]
Note that differentiation is continuous $\partial_k : W^{m,p}_\delta\to W^{m-1,p}_{\delta+1}$. 
Recall from \eqref{eq:W-multiplication} that  for $m>d/p$ and any $\delta_1,\delta_2\in\R$, pointwise multiplication of functions 
$(f,g)\mapsto fg$ defines a continuous map
\[
W^{m,p}_{\delta_1}\times W^{m,p}_{\delta_2} \to W^{m,p}_{\delta_1+\delta_2+\frac{d}{p}}.
\]
Since $Q$ involves only differentiation, pointwise products, and sums of elements of $\dd u$, we then conclude that \eqref{Q-map} 
is real analytic, and hence smooth.
\end{proof}

For future use, let us see what happens when $Q$ is conjugated with $\varphi\in{\mathcal D}^{m,p}_\gamma$:

\begin{Lem} \label{le:Q_phi} 
If $m>1+d/p$ and $\gamma+d/p>-1$, then the map $(\varphi,v)\mapsto R_\varphi\circ Q\circ R_{\varphi^{-1}}(v)$ is smooth
${\mathcal D}^{m,p}_\gamma \times W^{m,p}_\delta \to W^{m-1,p}_{2\delta+2+d/p}$ for any $\delta\in\R$.
\end{Lem}

\begin{proof}[Proof of Lemma \ref{le:Q_phi}]
The proof of Lemma \ref{le:Q_phi} follows from the formula
\[
R_\varphi\circ Q\circ R_{\varphi^{-1}}(v)={\rm tr} \left( [R_\varphi\circ\dd\circ R_{\varphi^{-1}}(v)]^2\right)
\]
and using Lemma \ref{le:nabla_phi}(a) to show $(\varphi,v)\mapsto R_\varphi\circ\dd\circ R_{\varphi^{-1}}(v)$ is smooth. 
\end{proof}
\noindent
Note that we will be applying Lemma \ref{le:Q_phi} below with $\gamma=\delta$. 

Now, to  use \eqref{Lap-p} to determine ${\rm p}$, we want to apply $\Delta^{-1}$ as found in Section \ref{sec:Laplacian} to $Q(u)$. 
In fact, since we only need $\nabla {\rm p}$ in \eqref{eq:euler}, we will now determine the mapping properties of 
$u\mapsto \nabla\circ\Delta^{-1}\circ Q(u)$ on $W^{m,p}_\delta$, which maps vector fields to vector fields. 
It follows from Lemma \ref{le:ND} that
\begin{subequations}\label{NablaDelta^{-1}}
\begin{equation}\label{NablaDelta^{-1}-a}
\nabla\circ\Delta^{-1} : W^{m-1,p}_{\kappa+1}\to W^{m,p}_{\kappa}
\end{equation}
is bounded provided $m\ge 1$ and 
\begin{equation}\label{NablaDelta^{-1}-b}
0<\kappa+d/p<d-1, \quad \kappa+d/p\ne 1.
\end{equation}
\end{subequations}
The strategy is now to choose $\kappa$ so that
$W^{m-1,p}_{2\delta+2+\frac{d}{p}} \subseteq  W^{m-1,p}_{\kappa+1}$ and
$W^{m,p}_{\kappa}\subseteq W^{m,p}_\delta$, as well as  \eqref{NablaDelta^{-1}-b}.
This can be done by choosing
\begin{equation}\label{est:kappa}
\max\big(0,\delta+d/p\big) < \kappa+d/p<\min\Big(d-1, 2\big(\delta+d/p\big)+1\Big)
\quad\hbox{and}\quad\kappa+d/p\ne 1.
\end{equation}

\begin{Prop}\label{pr:NablaDelta^{-1}Q} 
For $m>1+d/p$ and  $-1/2 < \delta+d/p < d-1$, the mapping
\begin{equation}\label{eq:nonlinearmap}
\nabla\circ\Delta^{-1}\circ Q : W^{m,p}_{\delta}\to W^{m,p}_\delta 
\end{equation}
is smooth.
\end{Prop}

\begin{proof}[Proof of Proposition \ref{pr:NablaDelta^{-1}Q}]
One can easily confirm that the condition $-1/2 < \delta+d/p < d-1$ implies
$\max\big(0,\delta+d/p\big)<\min\big(d-1,2(\delta+d/p)+1\big)$, so we can choose 
$\kappa$ to satisfy all conditions in \eqref{est:kappa}, and therefore 
$\nabla\circ\Delta^{-1}\circ Q : W_\delta^{m,p}\to W_\kappa^{m,p}\subseteq W_\delta^{m,p}$.
\end{proof}

The above analysis enables us to eliminate {\rm p} from \eqref{eq:euler} and to rewrite it as
\begin{equation}\label{eq:euler2}
\left\{
\begin{array}{l}
u_t+(u\cdot \nabla)\, u=\nabla\circ\Delta^{-1}\circ Q(u),\\
u|_{t=0}=u_0.
\end{array}
\right.
\end{equation}
Notice that the condition $\Div\,u=0$ for $t>0$ has also been eliminated. 
We have the following lemma.

\begin{Lem}\label{le:iff}  
Assume that $m>1+d/p$ and $-1/2<\delta+d/p<d-1$.
If the curve $u\in C\big([0,\tau],\df W^{m,p}_\delta\big)\cap C^1\big([0,\tau],\df W^{m-1,p}_\delta\big)$
satisfies the Euler equation \eqref{eq:euler} with pressure satisfying $\big|\nabla{\rm p}(x,t)\big|=o(1)$ as $|x|\to\infty$ 
for $t\in[0,\tau]$ then $u\in C\big([0,\tau],\df W^{m,p}_\delta\big)\cap C^1\big([0,T],\df W^{m-1,p}_\delta\big)$ 
satisfies \eqref{eq:euler2}. Conversely, if $u\in C\big([0,\tau],W^{m,p}_\delta\big)\cap C^1\big([0,T],W^{m-1,p}_\delta\big)$ 
satisfies \eqref{eq:euler2} with $u_0\in\df W^{m,p}_\delta$ then it satisfies \eqref{eq:euler} so that 
$\big|\nabla{\rm p}(x,t)\big|=o(1)$ as $|x|\to\infty$ and for any $t\in[0,\tau]$ we have that $p(t)=\Delta^{-1}\circ Q\big(u(t)\big)$ 
up to an additive constant.\footnote{Here $\Delta^{-1} : W^{m-1,p}_{\kappa+1}\to W^{m+1,p}_{\kappa-1}\slash\mathcal{N}_0$
with $0<\kappa+d/p<1$ (cf. \eqref{eq:D-isomorphism1}).}
\end{Lem}

\noindent The proof of this lemma follows the lines of the proof of Lemma 4.1 in \cite{McOwenTopalov3}, so we will not give the details.

Now we want to replace \eqref{eq:euler2} with an ordinary differential equation on the tangent bundle of 
${\mathcal D}^{m,p}_\delta$. The differential structure of ${\mathcal D}^{m,p}_\delta$ is inherited from $W^{m,p}_\delta$ 
in a natural way, so ${\mathcal D}^{m,p}_\delta$ may be viewed as a Banach manifold modeled on the Banach space 
$W^{m,p}_\delta$. This allows us to identify the tangent bundle $T{\mathcal D}^{m,p}_\delta$ with the product space:
\begin{equation}\label{T(D)}
T{\mathcal D}^{m,p}_\delta={\mathcal D}^{m,p}_\delta\times W^{m,p}_\delta.
\end{equation}
Next we define the {\it Euler vector field} \,${\mathcal E}$ on the tangent bundle ${\mathcal D}^{m,p}_\delta\times W^{m,p}_\delta$. 
For $\varphi\in {\mathcal D}^{m,p}_\delta$ and $v\in W^{m,p}_\delta$, we introduce
\begin{equation}
{\mathcal E}_2(\varphi,v)=R_\varphi \circ\nabla\circ\Delta^{-1}\circ Q\circ R_{\varphi^{-1}} (v).
\end{equation}
For $-1/2<\delta+d/p<d-1$ we can use Proposition \ref{pr:NablaDelta^{-1}Q} to conclude 
${\mathcal E}_2(\varphi,v) \in W^{m,p}_\delta$. So if we define  ${\mathcal E}(\varphi,v)=(\varphi, {\mathcal E}_2(\varphi,v))$, 
we obtain a map
\begin{equation}\label{E-vectorfield}
{\mathcal E}: {\mathcal D}^{m,p}_\delta\times W^{m,p}_\delta \to {\mathcal D}^{m,p}_\delta\times W^{m,p}_\delta.
\end{equation}
To have a unique integral curve, we need  ${\mathcal E}$ to be at least Lipschitz continuous on $T{\mathcal D}^{m,p}_\delta$. 
In fact, we will show ${\mathcal E}$ is smooth on $T{\mathcal D}^{m,p}_\delta$ for certain values of $\delta$.

\begin{Th}\label{Smoothness-Euler} 
The vector field \eqref{E-vectorfield} is smooth for $d\ge 2$, $m>1+d/p$, and $-1/2 < \delta+d/p < d-1$.
\end{Th}

\begin{proof}[Proof of Theorem \ref{Smoothness-Euler}]
We  factor ${\mathcal E}_2$ as follows:
\[
{\mathcal E}_2(\varphi,v)=\left(R_\varphi\circ\nabla\circ\Delta^{-1}\circ R_{\varphi^{-1}} \right)\circ
\left(R_\varphi\circ Q\circ R_{\varphi^{-1}}\right)(v).
\]
Recall from the proof of Proposition \ref{pr:NablaDelta^{-1}Q} the decomposition
\[
\begin{tikzcd}
W^{m,p}_\delta\arrow[r,"Q"]&W^{m-1,p}_{2\delta+2+d/p}\arrow[r,"\imath"]&
W^{m-1,p}_{\kappa+1}\arrow[r,"\nabla\circ\Delta^{-1}"]&
W^{m,p}_\kappa\arrow[r,"\imath"]& W^{m,p}_\delta,
\end{tikzcd}
\]
where $\imath$ denotes inclusion and $\kappa$ has been chosen to satisfy \eqref{est:kappa}. Conjugation by 
$R_\varphi$ of $Q : W^{m,p}_\delta \to W^{m-1,p}_{2\delta+2+d/p} $ is smooth by  Lemma \ref{le:Q_phi} 
(for any $\delta\in\R$). The conjugation by $R_\varphi$  of $\nabla\circ\Delta^{-1} : W^{m-1,p}_{\kappa+1} \to W^{m,p}_\kappa$  
is smooth by Proposition \ref{pr:ND-conjugate}.
\end{proof}

\medskip

The system of ordinary differential equations associated with the Euler vector field is
\begin{equation}\label{eq:euler3}
\left\{
\begin{array}{l}
(\dot\varphi,\dot v) = {\mathcal E}(\varphi,v),  \\
(\varphi,v)|_{t=0}=(\id,u_0).
\end{array}
\right.
\end{equation}
The relationship of $(\varphi,v)$ to the solution of \eqref{eq:euler2} is provided by the following.

\begin{Prop}\label{pr:bijection} 
The map 
\begin{equation}\label{bijection}
(\varphi,v)\mapsto u:=R_{\varphi^{-1}}(v)
\end{equation}
provides a continuous, bijective correspondence between solutions $(\varphi,v)\in C^1\big([0,\tau],T{\mathcal D}^{m,p}_\delta\big)$ of 
\eqref{eq:euler3} and solutions $u\in C\big([0,\tau],W^{m,p}_\delta\big)\cap C^1\big([0,\tau],W^{m-1,p}_\delta\big)$ of \eqref{eq:euler2}.
\end{Prop}

\noindent This proposition follows from Proposition \ref{pr:flowODE} and can be proved just like Lemma 7.1 in \cite{McOwenTopalov3}, 
so we will not repeat the details here.

\medskip

We can now prove Theorem \ref{th:main1} when $m>1+d/p$ and  $-1/2<\delta+d/p<d-1$. 

\begin{Th}\label{th:main1'}  
Assume  $m>1+d/p$, $1<p<\infty$, $d\ge 2$, and $-1/2<\delta+d/p<d-1$. Then, for any given $\rho>0$ there exists 
$\tau>0$ such that for any $u_0\in \df W^{m,p}_{\delta}$ with $\|u_0\|_{W^{m,p}_{\delta}}<\rho$ there exists a unique solution 
$u\in C\big([0,\tau],\df W^{m,p}_{\delta}\big)\cap C^1\big([0,\tau],\df W^{m-1,p}_{\delta}\big)$ of the Euler equation \eqref{eq:euler} 
such that $\big|\nabla{\rm p}(t,x)\big|=o(1)$ as $|x|\to\infty$ for $t\in[0,\tau]$. 
Moreover, for any fixed $t\in[0,\tau]$ the pressure ${\rm p}(t)$ is uniquely determined up to an additive constant. 
The solution depends continuously on the initial data $u_0\in\df W^{m,p}_\delta$.
\end{Th} 

\begin{proof}[Proof of Theorem \ref{th:main1'}]
Under the assumption that $\big|\nabla{\rm p}(t,x)\big|=o(1)$ as $|x|\to\infty$ we obtain from Lemma \ref{le:iff} 
and Proposition \ref{pr:bijection} that the Euler equation \eqref{eq:euler} is equivalent to 
the dynamical system \eqref{eq:euler3} on $T{\mathcal D}^{m,p}_\delta$.
We know by Theorem \ref{Smoothness-Euler} that ${\mathcal E}$ is smooth on $T{\mathcal D}^{m,p}_\delta$. 
Hence,  by the standard theory of ordinary differential equations in Banach spaces, for any $u_0\in W^{m,p}_\delta$ there exists 
$\tau=\tau(\|u_0\|_{W^{m,p}_\delta})>0$ and a unique solution $(\varphi,v)\in C^1\big([0,\tau],T{\mathcal D}^{m,p}_\delta\big)$ of 
\eqref{eq:euler3} that depends smoothly on the initial data. The uniqueness and the continuous dependence of $u$ on initial data 
then follows from the uniqueness and the continuous dependence of the solutions of \eqref{eq:euler3} on the initial data together with 
Lemma \ref{le:iff} and Proposition \ref{pr:bijection}. The statement on the pressure also follows from Lemma \ref{le:iff}.
\end{proof}

\section{Proofs of Theorem \ref{th:main2}, Propositions \ref{prop:asymptotic}, and \ref{prop:no_gain_no_loss}}\label{sec:theorem_main2}
We will deduce Theorem 1.2 using results in  \cite{McOwenTopalov3} concerning the well-posedness of the Euler equations in asymptotic spaces. 
We will give here a brief description of the asymptotic spaces $\A^{m,p}_{n,N;0}$ that we need; for more details, see Appendix \ref{sec:C} of 
this paper or the papers \cite{McOwenTopalov2} and  \cite{McOwenTopalov3}.

For integers $m>d/p$ and $N\geq n
\geq 0$, $\A^{m,p}_{n,N;0}$ is a Banach space whose elements are vector fields on $\R^d$ of the form
\begin{equation}\label{eq:asymptotic_expansion_with_log's}
u(x)=\chi(r)\left(\frac{a_n^0(\theta)+\cdots+a_n^n(\theta)(\log r)^n}{r^n}+\cdots+
\frac{a_N^0(\theta)+\cdots+a_N^N(\theta)(\log r)^N}{r^N}\right)+f(x).
\end{equation}
Here $a_k^j\in H^{m+1+N-k,p}(S^{d-1},\R^d\big) \subseteq C\big(S^{d-1},\R^d\big)$ for  $n\le k\le N$ and $0\le j\le k$; $H^{m,p}$ 
denotes the standard $L^p$ Sobolev space of order $m$. The remainder function $f(x)$ satisfies 
$f\in W^{m,p}_{\gamma_N}$ for a weight 
\begin{equation}\label{eq:gamma_N}
\gamma_N:=N+\gamma_0\quad\text{\rm where}\quad 0\leq\gamma_0+d/p<1,
\end{equation} 
which by (3b) implies $f(x)=o\big(r^{-N}\big)$ as $r\to\infty$. If $n=0$, we use the abbreviation $\A^{m,p}_{N;0}$ and we let 
$\ddf\A^{m,p}_{N;0}$ denote the closed subspace of divergence free vector fields in $\A^{m,p}_{N;0}$.
By $\A^{m,p}_N$ (resp. $\A^{m,p}_{n,N}$) we denote the closed subspace of $\A^{m,p}_{N;0}$ (resp. $\A^{m,p}_{n,N;0}$) that consists of
vector fields of the form \eqref{eq:asymptotic_expansion_with_log's} without log terms.
We refer to Appendix \ref{sec:C} for more details.

Actually, in \cite{McOwenTopalov3} and \cite{McOwenTopalov2}, strict inequality $0<\gamma_0+d/p<1$ was assumed in \eqref{eq:gamma_N}. 
In this case, it was shown in Theorem 1.1 in \cite{McOwenTopalov3} that for any $m>3+d/p$ and $\rho>0$ there exists $\tau>0$ such that 
for any $u_0\in\ddf \A^{m,p}_{N;0}$ with $\|u_0\|_{\A^{m,p}_{N;0}}<\rho$ there exists a unique solution of the Euler equation
\begin{equation}\label{eq:u->asymptotic_space}
u\in C\big([0,\tau],\ddf\A^{m,p}_{N;0}\big)\cap C^1\big([0,\tau],\ddf\A^{m-1,p}_{N;0}\big),
\end{equation}
that depends continuously on the initial data $u_0$.
We will use this  to prove our Theorem \ref{th:main2}.

\begin{proof}[Proof of Theorem \ref{th:main2}]
Let $N$ denote the integer part of $\delta+d/p>0$ and $\gamma_0:=\delta-N$. Then
$N\ge 0$ and $0\leq \gamma_0+d/p<1$. Since $\delta=\gamma_N$, $W_{\delta}^{m,p}$ is the remainder space for $\A^{m,p}_{N;0}$. 
In particular, $W_{\delta}^{m,p}$ is a closed subspace in $\A^{m,p}_{N;0}$.
Consequently, for any $\rho>0$ there exists $\tau>0$ such that for any $u_0\in\df W^{m,p}_\delta$ with $\|u_0\|_{W^{m,p}_\delta}<\rho$
the Euler equation \eqref{eq:euler} has a unique solution \eqref{eq:u->asymptotic_space} that depends continuously on the initial data 
in $W^{m,p}_\delta$. If $0<\gamma_0+d/p<1$ we refer as above to \cite[Theorem 1.1]{McOwenTopalov3}, while for $\gamma_0+d/p=0$ 
we refer to Proposition \ref{pr:Euler-n=1} in Appendix \ref{sec:C}.

We first assume $\delta+d/p\notin\Z$, so $0<\gamma_0+d/p<1$.
Since the case when $d=2$ follows from Corollary 1.1 and Corollary 1.2 in \cite{SultanTopalov},
we will concentrate our attention on the case when $d\ge 3$.
To see that the asymptotic terms appearing in the solution \eqref{eq:u->asymptotic_space} are of the form described 
in Theorem \ref{th:main2}, we argue as follows. Since $m>3+d/p$ it follows from \eqref{eq:u->asymptotic_space} and 
the Sobolev embedding theorem that $u\in C^1\big([0,\tau],C^2(\R^d,\R^d)\big)$. By applying the $\curl$ operator to \eqref{eq:euler} 
we then see that $u_t+L_u\omega=0$ where $L_u\omega$ denotes the Lie derivative of the vorticity form $\omega:=\dd(u^\flat)$ with 
respect to the time dependent vector field $u$. (Here $\dd$ denotes the exterior differentiation of the one form $u^\flat$ 
obtained by lowering the indices of the vector field $u$ with the help of the Euclidean metric on $\R^d$.) This implies that the pull-back
$\varphi(t)^*\big(\omega(t)\big)$ of the vorticity form $\omega(t)$ with respect to the flow map $\varphi(t) : \R^d\to\R^d$ generated by 
$u$ is independent of $t\in[0,\tau]$. Recall from Corollary 2.2 in \cite{McOwenTopalov3} that 
\begin{equation}\label{eq:the_flow_map}
\varphi\in C^1\big([0,\tau],\ddf\A D^{m,p}_{N;0}\big),
\end{equation}
where $\ddf\A D^{m,p}_{N;0}$ denotes the group of volume preserving asymptotic diffeomorphisms of $\R^d$ 
(see Section 2 in \cite{McOwenTopalov3}). 
Hence, for any $t\in[0,\tau]$ we have that\footnote{This is a re-expression of the conservation of vorticity theorem.}
\begin{equation}\label{eq:the_conservation_law}
\omega(t)=(\dd\psi(t))^T\,\big(\omega(0)\circ\psi(t)\big)\,(\dd\psi(t)),\quad\psi(t):=\varphi(t)^{-1},
\end{equation}
where $\dd\psi$ denotes the Jacobian matrix of $\psi(t) : \R^d\to\R^d$, 
\begin{equation}\label{eq:vorticity}
\omega=(\omega_{\alpha j})_{1\le\alpha, j\le d},\quad\omega_{\alpha j}=\frac{\partial u_j}{\partial x_\alpha}-\frac{\partial u_\alpha}{\partial x_j},
\end{equation}
is the matrix of the components of the vorticity form $\omega(t)$, and $(\cdot)^T$ denotes the transpose of a matrix.
Recall from Theorem 9.1 in \cite{McOwenTopalov3} that $\ddf\A D^{m,p}_{N;0}$ is a real analytic submanifold in the group
$\A D^{m,p}_{N;0}$ of asymptotic diffeomorphisms of $\R^d$. The elements of $\A D^{m,p}_{N;0}$ are 
$C^1$-diffeomorphisms of $\R^d$ of the form $\id+w$ where $w\in\A^{m,p}_{N;0}$.
Since $\ddf\A D^{m,p}_{N;0}$ (and $\A D^{m,p}_{N;0}$) is a topological group (cf. \cite[Corollary 2.1]{McOwenTopalov3}) 
we obtain from \eqref{eq:the_flow_map} that $\psi\in C\big([0,\tau],\ddf\A D^{m,p}_{N;0}\big)$. 
By combining this with \eqref{eq:vorticity}, the fact that $\omega(0)=\dd(u_0^\flat)\in W^{m-1,p}_{\delta+1}$, 
Corollary 6.1 {\rm (b)} in \cite{McOwenTopalov2},\footnote{Note that $\A D^{m,p}_{N;0}\subseteq\A D^{m,p}_0$ and 
the inclusion is bounded.} and the conservation law \eqref{eq:the_conservation_law}, we conclude from 
the pointwise multiplication properties of W-spaces (see \eqref{eq:W-multiplication} and \eqref{eq:W-derivative}) that
\begin{equation}\label{eq:vorticity_in_W}
\omega\in C\big([0,\tau],W^{m-1,p}_{\delta+1}\big).
\end{equation}
Since the solution \eqref{eq:u->asymptotic_space} is divergence free, we obtain from \eqref{eq:vorticity} that 
for any $1\le j\le d$,
\begin{equation*}
(\Div \omega)_j:=\sum_{\alpha=1}^d\partial_\alpha\omega_{\alpha j}
=\Delta u_j-\partial_j\big(\Div  u\big)=\Delta u_j
\end{equation*}
where $\partial_\alpha$ denotes the distributional partial derivative in the direction $x_\alpha$.
In view of the assumption that $\delta+d/p\notin\Z$, we then conclude that for any $t\in[0,\tau]$,
\begin{equation}\label{eq:biot-savart_law}
u(t)=\Delta^{-1}\Div \omega(t),
\end{equation}
where $\Delta^{-1} : W^{m-2,p}_{\delta+2}\to\A^{m,p}_N$ is the inverse of the Laplace operator given by
Proposition \ref{prop:inverting_the_laplace_operator} in Appendix \ref{sec:B}
(cf. Lemma A.3 {\rm (b)} in \cite{McOwenTopalov3}) and $\A^{m,p}_N\subseteq\A^{m,p}_{N;0}$ is the asymptotic space
without log terms. Note that by Proposition 1.1 in \cite{McOwenTopalov3} the leading term
$a_0$ in the asymptotic expansion \eqref{eq:asymptotic_expansion_with_log's}(with $n=0$) of the solution 
\eqref{eq:u->asymptotic_space} vanishes, since $u_0\in W^{m,p}_\delta$. 
This, together with Proposition \ref{prop:inverting_the_laplace_operator} in Appendix \ref{sec:B}, then implies that
\begin{equation}\label{eq:solution_pre-general_form}
u(t)=\chi(r)\sum_{d-2\le k<\delta+d/p}\frac{a_k(\theta,t)}{r^{k}}+f(t),\quad t\in[0,\tau],
\end{equation}
where $f(t)\in W^{m,p}_\delta$ and $a_k(\theta,t)$ with $d-2\le k\le N$ is an eigenfunction of the Laplace 
operator $-\Delta_S$ on the unit sphere $S^{d-1}$ with eigenvalue $\lambda_{k-d+2}=k(k-d+2)$.
In fact, it follows from Remark 1.3 in \cite{McOwenTopalov3} (or, alternatively, it can be deduced from
from \eqref{eq:euler2'} and the integral identity \eqref{eq:stokes3} below) that the three leading terms in 
\eqref{eq:solution_pre-general_form} do {\em not} appear, so $u(t)$ is of the form \eqref{eq:solution_general_form}. 
This completes the proof of items {\rm (a)} and {\rm (b)} of Theorem \ref{th:main2}.
The continuous dependence of \eqref{eq:solution_general_form} on the initial data $u_0\in\df W^{m,p}_\delta$ 
(as described in Remark \ref{rem:extended_well-posedness}) follows from the continuous dependence of 
the solution \eqref{eq:u->asymptotic_space} and the fact that $W^{m,p}_\delta$ is a closed subspace in $\A^{m,p}_{N;0}$.

Let us now prove the analyticity of the asymptotic coefficients $a_k$, $d+1\le k\le\delta+d/p$, in the case when 
$\delta+d/p\ge d+1$ and $\delta+d/p\notin\Z$. Take an integer $k\in\Z$ such that\footnote{Since $\delta+d/p\notin\Z$ we 
have that $k<\delta+d/p$.}
\begin{equation}\label{eq:k_range}
d+1\le k\le\delta+d/p
\end{equation} 
and $1\le l\le \nu(k')\equiv\dim\mathcal{H}_{k'}$ and set $k'\equiv k-d+2$.
It then follows from \eqref{eq:biot-savart_law} and Proposition \ref{prop:inverting_the_laplace_operator} in Appendix \ref{sec:B} 
that, for any given $t\in[0,\tau]$, the Fourier coefficient $\widehat{a}_{k';l}(t)$ in the Fourier expansion of $a_k(t)\in C(S^{d-1},\R^d)$ 
can be obtained from the integral formula
\begin{equation}\label{eq:a_{k';l}}
\widehat{a}_{k';l}(t)=-C_{k'}\int_{\R^d}H_{k';l}\Div \omega(t)\,\dd x,
\end{equation}
where $C_{k'}:=1/(2k'+d-2)$, $(\Div \omega)_j\equiv\sum_{\alpha=1}^d\partial_{\alpha}\omega_{\alpha j}$, 
and $H_{k';l}$ is a harmonic polynomial of degree $k'\ge 0$ (see Proposition \ref{prop:inverting_the_laplace_operator} for 
the precise definition of $H_{k';l}$). It follows from \eqref{eq:vorticity_in_W} that for any $t\in[0,\tau]$,
\begin{equation}\label{eq:w,dw}
\omega(t)\in W^{m-1,p}_{\delta+1}\quad\text{\rm and}\quad\Div\omega(t)\in W^{m-2,p}_{\delta+2}.
\end{equation}
Note that by H\"older's inequality
\begin{equation}\label{eq:W->L1}
W^{m,p}_\delta\subseteq L^1
\end{equation}
for $m\ge 0$ and $\delta+d/p>d$.
By combining \eqref{eq:w,dw} with the fact that $d+1\le k<\delta+d/p$ we see from \eqref{eq:W->L1} that we have enough decay to 
apply Stokes' theorem to the integral in \eqref{eq:a_{k';l}} and obtain from \eqref{eq:the_conservation_law} that
\begin{align}
\widehat{a}_{k';l}(t)&=C_{k'}\int_{\R^d}\big(\nabla H_{k';l}\big)^T\,\omega(t)\,\dd x
=C_{k'}\int_{\R^d}\big(\nabla H_{k';l}\big)^T\,\big[(\dd\psi(t))^T\,\big(\omega(0)\circ\psi(t)\big)\,(\dd\psi(t))\big]\,\dd x\nonumber\\
&=C_{k'}\int_{\R^d}
\big((\nabla H_{k';l})\circ\varphi(t)\big)^T\,\big[\big((\dd\varphi(t))^{-1})^T\,\omega(0)\,(\dd\varphi(t))^{-1}\big]\,\dd x,
\label{eq:a_k_analyticity}
\end{align}
where we changed the variables and used that $(\dd\psi)\circ\varphi=(\dd\varphi)^{-1}$ and that $\varphi(t)$ is volume preserving diffeomorphism 
by \eqref{eq:the_flow_map}. By taking $\tau>0$ smaller if necessary, we obtain from Proposition 9.2 in 
\cite{McOwenTopalov3} that the map
\[
(t,u_0)\mapsto\varphi(t;u_0),\quad [0,\tau]\times B_{\df W^{m,p}_\delta}(\rho)\to\ddf\A D^{m,p}_{N;0},
\]
is analytic. By combining this with \eqref{eq:a_k_analyticity}, the properties \eqref{eq:W-multiplication}, \eqref{eq:W-derivative}, and 
the fact that $H_{k',l}$ is a polynomial, we conclude that $\widehat{a}_{k';l} : [0,\tau]\times B_{\df W^{m,p}_\delta}(\rho)\to\R$ is analytic.
The global existence of the solution \eqref{eq:u-general_form} in the case when $d=2$ and $\delta+2/p>0$ is not integer
follows from Corollary 1.1 and 1.2 in \cite{SultanTopalov}.
This completes the proof of Theorem \ref{th:main2} in the case when $\delta+d/p>0$ is not integer.
Note that the argument above cannot be readily applied in the case when $\delta+d/p>0$ is integer since 
Proposition \ref{prop:inverting_the_laplace_operator} in Appendix \ref{sec:B} does {\em not} hold for 
integer values of $\delta+d/p$.

\medskip
Before turning to the case when $\delta+d/p\in\Z$ note that for any $t\in[0,\tau]$ we have
$\dd\varphi(t)=\id+\dd w(t)$ where $\dd w(t)\in\A^{m-1,p}_{1,N+1;-1}$ (cf. Appendix \ref{sec:C}).
This implies that for any integer $k$ such that $d+1\le k\le\delta+d/p$ the integrand $I_{k'}\big(\varphi(t),\omega(0)\big)$ in 
\eqref{eq:a_k_analyticity} can be written as
\begin{equation}\label{eq:I_k'-decomposition}
I_{k'}\big(\varphi(t),\omega(0)\big)=\big(\nabla H_{k';l}\big)^T\omega(0)+R_{k'}\big(\varphi(t),\omega(0)\big)
\end{equation}
where $\omega(0)\equiv\dd u_0^\flat$ and $R_{k'}\big(\varphi(t),\omega(0)\big)\in W^{m-1,p}_{(\delta+1)-(k'-1)+1}\subseteq L^1$  
by \eqref{eq:W->L1}. Moreover, for any $\delta\in\R$ with $\delta+d/p\ge d+1$ and for any integer $d+1\le k\le\delta+d/p$ the map
\begin{equation}\label{eq:R_k'}
(\varphi,u_0)\mapsto R_{k'}(\varphi,\dd u_0^\flat),\quad
\A D^{m,p}_{N;0}\times B_{{\df W}^{m,p}_\delta}(\rho)\to L^1,
\end{equation}
is analytic. By combining this with \eqref{eq:a_k_analyticity} we conclude that for any {\em non-integer} $\delta+d/p\ge d+1$ 
for any $d+1\le k\le\delta+d/p$ and for any $1\le l\le \nu(k')$ we have that
\begin{equation}\label{eq:a_k_analyticity'}
\widehat{a}_{k';l}(t;u_0)=C_{k'}\int_{\R^d}R_{k'}\big(\varphi(t;u_0),\dd u_0^\flat\big)\,\dd x
\end{equation}
where we used that for $d+1\le k<\delta+d/p$ we have from \eqref{eq:W->L1} that 
$\big(\nabla H_{k';l}\big)^T(\dd u_0^\flat)\in W^{m-1,p}_{\delta+1-k'+1}\subseteq L^1$ 
and, by the Stokes' theorem, \eqref{eq:vorticity}, and the fact that $H_{k',l}$ is a harmonic polynomial,
\begin{align*}
\int_{\R^d}\big(\nabla H_{k';l}\big)^T(\dd u_0^\flat)\,\dd x&=
\left(\int_{\R^d}\sum_{\alpha=1}^d\big(\partial_\alpha H_{k';l}\big)
\big(\partial_\alpha u_{0j}-\partial_j u_{0\alpha}\big)\,\dd x\right)_{1\le j\le d}\\
&=-\int_{\R^d}\big(\Delta H_{k';l}\big)\,u_0\,\dd x+\int_{\R^d} H_{k';l}\nabla\big(\Div u_0\big)\,\dd x=0
\end{align*}
for $u_0\in{\df W}^{m,p}_\delta$.

\begin{Rem}\label{rem:k=delta+d/p}
Note that for $\varphi(t;u_0)\in\A D^{m,p}_{N;0}$ (with $N$ the integer part of $\delta+d/p$) the expression on the right side of  
\eqref{eq:a_k_analyticity'} is well defined for all values of $\delta+d/p\ge d+1$ and for any $d+1\le k\le\delta+d/p$. 
In particular, it is well defined for $\delta+d/p\in\Z$ and $k=\delta+d/p$, and depends analytically on $u_0\in{\ddf W}^{m-1,p}_\delta$.
(In contrast, $\dd u_0^\flat\in W^{m-1,p}_{\delta+1}$ implies that the term $\big(\nabla H_{k';l}\big)^T(\dd u_0^\flat
)$ in \eqref{eq:I_k'-decomposition} belongs to $W^{m-1,p}_{(\delta+1)-k'+1}$ that is not a subset in $L^1$ for 
$\delta+d/p\in\Z$ and $k=\delta+d/p$. In particular, \eqref{eq:a_k_analyticity} is {\em not} well defined for 
$\delta+d/p\in\Z$ and $k=\delta+d/p$.)
\end{Rem}

\medskip

With this preparation, we now turn to the case when $\delta+d/p\in\Z$ so $\delta+d/p\geq 1$ and $\gamma_0+d/p=0$;
we need to include $d=2$ since integral values of $\delta+d/p$ were not studied in \cite{SultanTopalov}.
Since $\gamma_N=\delta$, we see as above that $W^{m,p}_\delta$ is the remainder space for the asymptotic space 
$\A^{m,p}_{1,N;0}$. In particular, $W^{m,p}_\delta$ is a closed subspace in $\A^{m,p}_{1,N;0}$. By Proposition \ref{pr:Euler-n=1}, 
for any $\rho>0$ there exists $\tau>0$ such that for any $u_0\in{\df W}^{m,p}_\delta\subseteq\A^{m,p}_{1,N;0}$ with 
$\|u_0\|_{W^{m,p}_\delta}<\rho$ there exists a unique solution of the Euler equation
\[
u\in C\big([0,\tau],\A^{m,p}_{1,N;0}\big)\cap C^1\big([0,\tau],\A^{m-1,p}_{1,N;0}\big)
\]
and $\varphi\in C^1\big([0,\tau],\ddf\A D^{m,p}_{1,N;0}\big)$ (cf.\ Corollary 2.2 in \cite{McOwenTopalov3}) that depend continuously on 
the initial data $u_0\in{\df W}^{m,p}_\delta$ with $\|u_0\|_{W^{m,p}_\delta}<\rho$.
Now, take $\delta'>\delta$ so that $\delta'-\delta<1$. Then, $\delta'+d/p\notin\Z$ and the space ${\df W}^{m,p}_{\delta'}$ is dense in 
${\df W}^{m,p}_\delta$. Let $(u_{0j})_{j\ge 1}$ be a sequence of initial data in ${\df W}^{m,p}_{\delta'}$ such that
$\|u_{0j}\|_{W^{m,p}_\delta}<\rho$ and
\begin{equation}\label{eq:u_j0->u_0}
u_{0j}\stackrel{W^{m,p}_\delta}{\longrightarrow} u_0\quad\text{\rm as}\quad j\to\infty.
\end{equation}
Let 
\begin{equation}\label{eq:u_j}
u_j\in C\big([0,\tau],\A^{m,p}_{1,N;0}\big)\cap C^1\big([0,\tau],\A^{m-1,p}_{1,N;0}\big)
\end{equation}
be the corresponding solutions of the Euler equation in $\A^{m,p}_{1,N;0}$ and let
\begin{equation}\label{eq:phi_j}
(\varphi_j,v_j)\in C^1\big([0,\tau],\ddf\A D^{m,p}_{1,N,0}\times\A^{m,p}_{1,N;0}\big)
\end{equation}
be the integral curve of the smooth Euler vector field \eqref{Vectorfield-AD} (cf. Appendix \ref{sec:C}) with initial data 
$(\id,u_0)\in\ddf\A D^{m,p}_{1,N;0}\times\A^{m,p}_{1,N;0}$.
We then conclude from \eqref{eq:the_conservation_law}, the fact that $\A D^{m,p}_{1,N;0}$ is a topological group, 
$\A D^{m,p}_{1,N;0}\subseteq\A D^{m,p}_0$, Corollary 6.1 (b) in \cite{McOwenTopalov2}, 
the properties \eqref{eq:W-multiplication}, \eqref{eq:W-derivative}, and
$\omega_{0j}\equiv\dd\big(u_{0j}^\flat\big)\in W^{m-1,p}_{\delta'+1}$, that for any $j\ge 1$,
\[
\omega_{0j}\equiv\dd\big(u_{0j}^\flat\big)\in C\big([0,\tau],W^{m-1,p}_{\delta'+1}\big).
\]
Since $\delta'+d/p>0$ and $\delta'+d/p\notin\Z$, we then obtain from $u_j(t)=\Delta^{-1}\Div\omega_j(t)$, 
where $\omega_j(t)\equiv\dd\big(u_j^\flat(t)\big)$ and $t\in[0,\tau]$, and Proposition \ref{prop:inverting_the_laplace_operator} in 
Appendix \ref{sec:B} (see formula (86), and Proposition 3.3 all in \cite{SultanTopalov} for the case when $d=2$), 
that for any given $t\in[0,\tau]$ the solution $u_j(t)$ has an asymptotic expansion of the form
\begin{equation}\label{eq:u_j-asymptotics}
u_j(t)=\chi(r)\!\sum_{d+1\le k\le\delta+d/p}\!\frac{a_{kj}(\theta,t)}{r^{k}}\,+\,f_j(t),\quad f_j(t)\in W^{m,p}_\delta,
\end{equation}
with asymptotic coefficients as described in Theorem \ref{th:main2}.
In view of the continuous dependence of \eqref{eq:phi_j} on the initial data we then conclude from
\eqref{eq:u_j0->u_0} that for any given $t\in[0,\tau]$,
\begin{equation}\label{eq:u_j,phi_j,converge}
\begin{tikzcd}
u_j(t)\arrow[r, "\A^{m,p}_{1,N;0}"]&u(t)
\end{tikzcd}
\quad\text{\rm and}\quad
\begin{tikzcd}
\varphi_j(t)\arrow[r, "\A D^{m,p}_{1,N;0}"]&\varphi(t)
\end{tikzcd}
\quad\text{\rm as}\quad j\to\infty.
\end{equation}
This together with \eqref{eq:u_j-asymptotics} and the definition of the norm in $\A^{m,p}_{1,N;0}$ (cf. \cite[Section 2]{McOwenTopalov3}) 
then implies that 
\begin{tikzcd}
f_j\arrow[r, "W^{m,p}_\delta"]&f
\end{tikzcd}
and  $a_{kj}\stackrel{L^\infty}{\longrightarrow}a_k$ as $j\to\infty$ and hence,
\begin{equation}\label{eq:u-general_form}
u(t)=\chi(r)\!\sum_{d+1\le k\le\delta+d/p}\!\frac{a_k(\theta,t)}{r^{k}}\,+\,f(t),\quad f(t)\in W^{m,p}_\delta.
\end{equation}
Since $a_{kj}\stackrel{L^\infty}{\longrightarrow}a_k$ as $j\to\infty$ and since $a_{kj}\in C\big(S^{d-1},\R^d\big)$ is an eigenfunction of 
the Laplace operator $-\Delta_S$ on $S^{d-1}$ with eigenvalue $\lambda_{k'}=k'(k'+d-2)$, we conclude that the limit function $a_k$ 
on the sphere is again an eigenfunction of $-\Delta_S$ with the same eigenvalue. The well-posedness of the solution $u$ in the sense of 
\eqref{eq:extended_well-posedness} follows from the continuous dependence of $u$ on the initial data in the asymptotic space $\A^{m,p}_{1,N;0}$.
Moreover, it follows from \eqref{eq:u_j,phi_j,converge} and Remark \ref{rem:k=delta+d/p} that for any $d+1\le k\le\delta+d/p$, 
$1\le l\le \nu(k')$, and for any $t\in[0,\tau]$ the Fourier coefficient $\widehat{a}_{k';l}(t)$ of the asymptotic coefficient $a_k(t)$ in 
\eqref{eq:u-general_form} is given by \eqref{eq:a_k_analyticity'}. The analyticity of the asymptotic coefficients $a_k$, $d+1\le k\le\delta+d/p$, 
then follows from \eqref{eq:a_k_analyticity'} and the analyticity of the map \eqref{eq:R_k'}. Finally, the global existence in the case when $d=2$ and 
$\delta+2/p>0$ is an integer follows from Proposition \ref{prop:global_existence_d=2} in Appendix \ref{sec:D} and the approximation argument 
applied above.
\end{proof}


The proof of Proposition \ref{prop:asymptotic} is based on the following non-vanishing lemma.

\begin{Lem}\label{lem:non-trivial_ asymptotic}
Assume that $m>3+d/p$, $1<p<\infty$. Then for any $k'\ge 0$, $1\le j\le d$, and for any $\varepsilon>0$ there exists 
a homogeneous harmonic polynomial $H_{k'}^j(x)$ of degree $k'$ in $x\in\R^d$ and a divergence free vector field 
$u_0\in C_c^\infty(\R^d)$ with support in the annulus $\varepsilon<|x|<2\varepsilon$ such that\footnote{No summation over $j$ in 
the formula is assumed.}
\[
M_{k'}^j(u_0):=\int_{\R^d}H_{k'}^j\partial_j\big(Q(u_0)\big)\,\dd x\ne 0 
\]
where $Q(u_0)\equiv\tr\big(\dd u_0\big)^2$.
\end{Lem}

\begin{proof}[Proof of Lemma \ref{lem:non-trivial_ asymptotic}]
Assume that $m>3+d/p$, $1<p<\infty$. 
Let $P(x)$ be a homogeneous harmonic polynomial in $x\in\R^d$ of degree $k'\ge 0$.
Then, for any initial data with compact support $u_0\in C^\infty_c$ and for any given index $1\le j\le d$ we have
\begin{equation}\label{eq:stokes1}
M_{k'}^j(u_0)=\int_{\R^d}P\,\partial_j\big(Q(u_0)\big)\,\dd x=-\int_{\R^d}(\partial_j P)\,Q(u_0)\,\dd x
\end{equation}
where, by \eqref{def:Q},
\begin{equation}\label{def:Q_0}
Q(u_0)=\tr\big(\dd u_0\big)^2=\Div \big(u_0\cdot\nabla u_0\big).
\end{equation}
Further, by \eqref{eq:stokes1}, \eqref{def:Q_0}, and the Stokes' theorem,
\begin{align}
M_{k'}^j(u_0)&=-\int_{\R^d}(\partial_j P)\,\Div \big(u_0\cdot\nabla u_0\big)\,\dd x
=\int_{\R^d}\sum_{1\le\alpha,\beta\le d}\big(\partial_\alpha\partial_j P\big)\big(\partial_\beta u_{0\alpha}\big)u_{0\beta}\,
\dd x\nonumber\\
&=\int_{\R^d}\sum_{1\le\alpha,\beta\le d}\big(\partial_\alpha\partial_j P\big)\,\partial_\beta \big(u_{0\alpha}u_{0\beta}\big)\,\dd x
=-\int_{\R^d}\sum_{1\le\alpha,\beta\le d}\big(\partial_\beta\partial_\alpha\partial_j P\big)\,\big(u_{0\alpha}u_{0\beta}\big)\,
\dd x\label{eq:stokes3}
\end{align}
where we used that $\sum_{\beta=1}^d\partial_\beta u_{0\beta}\equiv\Div  u_0=0$ and the identity
$\Div (f X)=(\dd f)(X)+f\Div  X$ that holds for any $C^\infty$-smooth vector field $X$ and a scalar function $f$ on $\R^d$.
In particular, we confirm from \eqref{eq:stokes3} that 
\begin{equation*}
M_{k'}^j(u_0)=0\quad\text{\rm for}\quad 0\le k'\le 2.
\end{equation*}
Let us now assume that $k'\ge 3$.
Then, we fix the indexes $1\le\alpha,\beta\le d$, $\alpha\ne\beta$, and set
\begin{equation}\label{eq:u_0_in_hamiltonian_form}
u_0:=-(\partial_\beta H)\,\partial_\alpha+(\partial_\alpha H)\,\partial_\beta
\end{equation}
where the (scalar) Hamiltonian $H\in C_c^\infty$ will be specified later.  
Note that the vector field $u_0$ in \eqref{eq:u_0_in_hamiltonian_form} is automatically divergence free.
It then follows from \eqref{eq:stokes3} and \eqref{eq:u_0_in_hamiltonian_form} that for any choice of the homogeneous 
harmonic polynomial $P$ of degree ${k'}\ge 3$ and for any choice of the Hamiltonian $H\in C_c^\infty$ we have that
\begin{equation}\label{eq:momentum_simplified}
M_{k'}^\alpha(u_0)=\int_{\R^d}\Big\{2\big(\partial_\alpha^2\partial_\beta P\big)\,(\partial_\alpha H)(\partial_\beta H)
-\big(\partial_\alpha\partial_\beta^2 P\big)\,(\partial_\alpha H)^2
-\big(\partial_\alpha^3P\big)\,(\partial_\beta H)^2\Big\}\,\dd x
\end{equation}
where $u_0$ is given by \eqref{eq:u_0_in_hamiltonian_form}.
It is useful to introduce the complex variables $z:=x_\alpha+i x_\beta$ and $\z:=x_\alpha-i x_\beta$ and the Cauchy operators
$\partial_z:=\frac{1}{2}\big(\partial_\alpha-i\partial_\beta\big)$ and $\partial_{\z}:=\frac{1}{2}\big(\partial_\alpha+i\partial_\beta\big)$.
Then, $\partial_\alpha=\partial_z+\partial_\z$, $\partial_\beta=i\big(\partial_z-\partial_\z\big)$ and hence
\begin{equation}\label{eq:H-complex}
\partial_\alpha H=H_z+H_\z\quad\text{and}\quad\partial_\beta H=i\big(H_z-H_\z\big)
\end{equation} 
where $H_z\equiv\partial_zH$ and $H_\z\equiv\partial_\z H$. 
We now set
\begin{equation}\label{eq:P}
P(x):=z^{k'}+\z^{k'},\quad k'\ge 3,
\end{equation}
and note that $P$ is a harmonic polynomial (since we can write $\Delta=4\partial_z\partial_{\z}+\sum_{\mu\ne\alpha,\beta}\partial_\mu^2$). 
The integrand in \eqref{eq:momentum_simplified} can then be written as
\begin{align}
&2\big(\partial_\alpha^2\partial_\beta P\big)\,(\partial_\alpha H)(\partial_\beta H)
-\big(\partial_\alpha\partial_\beta^2 P\big)\,(\partial_\alpha H)^2
-\big(\partial_\alpha^3P\big)\,(\partial_\beta H)^2=\nonumber\\
&=2i\big(\partial_\alpha^2\partial_\beta P\big)\big(H_z^2-H_\z^2\big)
-\big(\partial_\alpha\partial_\beta^2 P\big)\Big(\big(H_z+H_\z\big)^2+\big(H_z-H_\z\big)^2\Big)\nonumber\\
&=4 i\partial_\alpha\partial_\beta\big(\partial_\z P\big) H_z^2-4i\partial_\alpha\partial_\beta\big(\partial_z P\big) H_\z^2\nonumber\\
&=4\Big(\big(\partial_\z^3 P\big)H_z^2+\big(\partial_z^3 P\big)H_\z^2\Big)\label{eq:integrand}
\end{align}
where we used \eqref{eq:H-complex}, \eqref{eq:P}, and the fact that $\partial_\alpha^2P=-\partial_\beta^2P$. 
We now choose
\begin{equation}\label{eq:H}
H(x):=R(x)\,a(\varrho)+b(\varrho_1,\varrho_2),\quad\varrho:=|x|^2,\quad\varrho_1:=z\z,\quad
\varrho_2:=|x|^2-|z|^2,
\end{equation}
where $R(x)$ is a homogeneous polynomial in $x\in\R^d$ and $a(\varrho)$ and $b(\varrho_1,\varrho_2)$ are $C^\infty$
functions of their arguments $a : \R\to\R$, $b : \R^2\to\R$, such that, when considered as functions of $x\in\R^d$, they have
non-empty support in the annulus $\varepsilon\le|x|\le 2\varepsilon$ for 
a given $\varepsilon>0$. 
Then,
\begin{equation}\label{eq:HH}
H_z=a R_z+a'\z R+b'_{\varrho_1}\z
\end{equation}
where $a'_\varrho$ denotes the derivative of $a\in C^\infty_c(\R)$ and $b'_{\varrho_1}$ denotes the partial derivative of 
$b(\varrho_1,\varrho_2)$ with respect to $\varrho_1$. 
Further, we set in \eqref{eq:H},
\begin{equation}\label{eq:R}
R(x):=z^\ell+\z^\ell,\quad\ell\ge 2,
\end{equation}
and then obtain from \eqref{eq:HH} that
\begin{equation}\label{eq:HH1}
H_z^2=\big(2\ell a b'_{\varrho_1}|z|^2+2 a' b'_{\varrho_1}|z|^4\big) z^{\ell-2}+\ldots
\end{equation}
where $\ldots$ stand for a sum of terms of the form 
\begin{equation}\label{eq:HH2}
c(\varrho_1,\varrho_2)\,|z|^\mu z^{2\ell-2}\quad\text{\rm or}\quad
c(\varrho_1,\varrho_2)\,|z|^\mu\z^\nu,\quad\nu\ge 2,\quad\mu\in\Z_{\ge 0},
\end{equation}
where $c(\varrho_1,\varrho_2)$ (different for each term) has support in the annulus $\varepsilon\le|x|\le 2\varepsilon$. 
Let us now set $\ell:=k'-1$, $k'\ge 3$. Then, it follows from  \eqref{eq:momentum_simplified}, \eqref{eq:P}, 
\eqref{eq:integrand}, \eqref{eq:HH1}, and \eqref{eq:HH2}, that
\begin{align}
M_{k'}^\alpha(u_0)&=4\int_{\R^{d-2}}\left(\int_{\R^2}
\Big\{\big(\partial_{\bar z}^3 P\big)H_z^2+\big(\partial_z^3 P\big)H_{\bar z}^2\Big\}\,\dd x_\alpha \dd x_\beta\right)\,\dd x'\nonumber\\
&=C\,\int_{\R^{d-2}}
\left(\int_0^\infty\int_0^{2\pi}b'_{\varrho_1}(\varrho_1,\varrho_2)\Big(\ell a(\varrho)+r^2a'_\varrho(\varrho)\Big)r^{2\ell-1}\,
\dd\theta\,\dd r\right)\,\dd x'\nonumber\\
&=C\pi\,\int_{\R^{d-2}}
\left(\int_0^\infty b'_{\varrho_1}(\varrho_1,\varrho_2)\Big(\ell a(\varrho)+\varrho_1 a'_\varrho(\varrho)\Big)\varrho_1^{\ell-1}\,
\dd\varrho_1\right)\,\dd x'\nonumber\\
&=C\pi\,\int_{\R^{d-2}}\left(\int_0^\infty 
\partial_{\varrho_1}\big(b(\varrho_1,\varrho_2)\big)\partial_{\varrho_1}\big(a(\varrho)\varrho_1^\ell\big)\,\dd\varrho_1
\right)\,\dd x'\label{eq:M-final}
\end{align}
where $\varrho\equiv|x|^2=\varrho_1+\varrho_2$, $\varrho_1= r^2$, $\varrho_2=|x'|^2$, and $C$ is a non-zero constant 
depending on $k'\ge 3$. Finally, by choosing 
\begin{equation}\label{eq:b}
b(\varrho_1,\varrho_2):=a(\varrho)\varrho_1^\ell
\end{equation} 
we obtain from \eqref{eq:M-final} that for any $k'\ge 3$ and for any $1\le\alpha\le d$ and 
\[
\beta:=\left\{
\begin{array}{l}
d,\quad\alpha\ne d\\
1,\quad\alpha=d
\end{array}
\right.
\]
we have that $M_{k'}^\alpha(u_0)\ne 0$ where $u_0$ is given by \eqref{eq:u_0_in_hamiltonian_form} 
with Hamiltonian
\begin{equation*}
H(x)=\Big(\big(z^{k'-1}+\z^{k'-1}\big)+|z|^{2k'-2}\Big) a\big(\varrho\big),\quad k'\ge 3,
\end{equation*}
by \eqref{eq:H}, \eqref{eq:R}, and \eqref{eq:b}. 
Finally, we take $\alpha=j$ and set $H_{k'}^j(x):=P(x)$ with $P(x)=z^{k'}+\z^{k'}$. 
This completes the proof of the lemma.
\end{proof}

\medskip

\begin{proof}[Proof of Proposition \ref{prop:asymptotic}]
Assume that $m>3+d/p$, $1<p<\infty$. Let us first consider the case when $\delta+d/p\ge d+1$ and $\delta+d/p\notin\Z$.
Since the subcase when $d=2$ follows from Proposition 1.2 (ii) in \cite{SultanTopalov}, we will concentrate our attention on
the case when $d\ge 3$. Take $d+1\le k<\delta+d/p$ and $1\le j\le d$. By Lemma \ref{lem:non-trivial_ asymptotic} there exists a homogeneous 
harmonic polynomial $H_{k'}^j$ and a divergence free vector field $u_0\in C^\infty_c$ such that (see \eqref{eq:stokes3})
\begin{equation}\label{eq:M=/=0}
M_{k'}^j(u_0)\equiv\int_{\R^d}H_{k'}^j\partial_j\,\big(Q(u_0)\big)\,\dd x=
-\int_{\R^d}\sum_{1\le\alpha,\beta\le d}\big(\partial_\beta\partial_\alpha\partial_j H_{k'}^j\big)\,\big(u_{0\alpha}u_{0\beta}\big)\,
\dd x
\end{equation}
does not vanish. 
Since $d+1\le k<\delta+d/p$, we conclude from \eqref{eq:W->L1} that 
$\big(\partial_\beta\partial_\alpha\partial_j H_{k'}^j\big)\,\big(u_{0\alpha}u_{0\beta}\big)\in L^1$ 
for any $u_0\in W^{m,p}_\delta$ and for any $1\le\alpha,\beta\le d$. 
Hence, the right side in \eqref{eq:M=/=0} defines a non-vanishing bounded quadratic form on $W^{m,p}_\delta\to\R$.
In particular, there exists an open dense set $\mathcal{N}$ in $\df W^{m,p}_\delta$ such that $M_{k'}^j(u_0)\ne 0$
for any $u_0\in\mathcal{N}$. Let us now take $u_0\in\mathcal{N}$ (not necessarily the one from Lemma \ref{lem:non-trivial_ asymptotic})
and let $u$ be the solution \eqref{eq:solution_general_form} of the Euler equation given by Theorem \ref{th:main2} {\rm (b)} with
initial data $u_0$. Let $N$ be the integer part of $\delta+d/p$. Then, by Lemma 4.1 in \cite{McOwenTopalov3}, the solution
$u\in C\big([0,\tau],\A^{m,p}_{N;0}\big)\cap C^1\big([0,\tau],\A^{m-1,p}_{N;0}\big)$ satisfies the equation
\begin{equation}\label{eq:euler2'}
u_t+\big(u\cdot\nabla\big) u=\Delta^{-1}\big(\nabla\circ Q(u)\big),\quad u|_{t=0}=u_0.
\end{equation}
By taking $t=0$ in \eqref{eq:euler2'} and then comparing the asymptotic terms in \eqref{eq:euler2'} we obtain from 
Proposition \ref{prop:inverting_the_laplace_operator} in Appendix \ref{sec:B} (cf. \eqref{eq:the_integral_formula}) and 
$(u_0\cdot\nabla) u_0\in W^{m-1,p}_\delta$ that
\begin{equation}\label{eq:a_k^j-t-derivative}
\big\langle\big(\partial_t a_k^j\big)\big|_{t=0},h_{k'}^j\big\rangle_{L^2(S^{d-1},\R)}=
C_{d,k}\int_{\R^d}H_{k'}^j\partial_j\big(Q(u_0)\big)\,\dd x\equiv C_{d,k}\,M_k^j(u_0)
\end{equation}
where $\big(\partial_t a_k^j\big)\big|_{t=0}$ is the $t$-derivative of the $j$-th component $a_k^j$ of the $k$-th asymptotic
coefficient of the expansion \eqref{eq:solution_general_form} of $u$, $C_{d,k}>0$ is a constant, and 
$h_{k'}^j(\theta):=H_{k'}^j(x)/r^{k'}$, $k'\equiv k-d+2\ge 3$.
It now follows from \eqref{eq:a_k^j-t-derivative} that the Fourier coefficient of $\big(\partial_t a_k^j\big)\big|_{t=0}$ corresponding to 
the spherical harmonic $h_{k'}^j$ does {\em not} vanish. 
(We choose such an orthonormal basis in the eigenspace $\mathcal{H}_{k'}$ of the Laplace operator $-\Delta_S$ on the unit sphere 
$S^{d-1}$ with eigenvalue $\lambda_{k'(k'+d-2)}$ (cf. Proposition \ref{prop:inverting_the_laplace_operator}) that includes 
the normalized eigenfunction $h_{k'}^j$.) By combining this with the fact that the $j$-th component 
$a_k^j(t)\in C\big(S^{d-1},\R\big)$, $t\in[0,\tau]$, of the asymptotic coefficient $a_k$ of the solution \eqref{eq:solution_general_form} 
is analytic in time (cf. Theorem \ref{th:main2}) we complete the proof of the proposition in the case when $\delta+d/p\notin\Z$.

Finally, consider the case when $\delta+d/p\ge d+1$ is an integer.
Arguing by approximation, we conclude from the continuous dependence of $u$ on the initial data in $W^{m,p}_\delta$
that the integral formula \eqref{eq:a_k^j-t-derivative} continues to hold. Since the quadratic form \eqref{eq:M=/=0} is bounded in 
$W^{m,p}_\delta$, the arguments above show that the proposition also holds in the case when $\delta+d/p\in\Z$. 
The case when $d=2$ follows in the same way from Lemma 4.3 in \cite{SultanTopalov} and Proposition \ref{pr:Euler-n=1} in Appendix \ref{sec:C}.
\end{proof}

Let us now prove Proposition \ref{prop:no_gain_no_loss}.

\begin{proof}[Proof of Proposition \ref{prop:no_gain_no_loss}]
For a fixed weight $\delta_0+d/p>0$ the independence of the interval of existence $[0,\tau]$, $\tau>0$, 
on the choice of the regularity exponent $m\ge m_0$ follows as in Proposition 4.1 in \cite{XuTopalov} (see \cite{EM} for the original argument)
and the analyticity of the map (71) in \cite{McOwenTopalov3}. The independence of the interval of existence on the choice of the weight 
$\delta\ge\delta_0$, for a given regularity exponent $m\ge m_0$, can then be deduced from the preservation of vorticity 
\eqref{eq:the_conservation_law}, as in the proof of Theorem \ref{th:main2}. We will omit the details of this proof.
\end{proof}

Finally, we prove Corollary \ref{coro:schwartz_data} stated in the Introduction.

\begin{proof}[Proof of Corollary \ref{coro:schwartz_data}]
For any $k\ge d+1$ and $1\le j\le d$ consider the homogeneous harmonic polynomial $H_{k'}^j$ given by Lemma \ref{lem:non-trivial_ asymptotic}.
It follows from \eqref{eq:a_k^j-t-derivative} that for any initial data $u_0\in\df\Sz$ the Fourier coefficient of $\big(\partial_t a_k^j\big)\big|_{t=0}$ 
corresponding to the spherical harmonic $h_{k'}^j(\theta)\equiv H_{k'}^j(x)/r^{k'}$ equals $C_{d,k}\,M_{k'}^j(u_0)$
where $k'\equiv k-d+2$ and $C_{d,k}>0$ is a constant. In view of \eqref{eq:M=/=0}, $M_{k'}^j : W^{m,p}_\delta\to\R$ is a bounded quadratic 
form on $W^{m,p}_\delta$ for any $m>3+d/p$, $1<p<\infty$, and $\delta+d/p>d+1$. 
This, together with Lemma \ref{lem:non-trivial_ asymptotic}, then implies that for any $k\ge d+1$ and $1\le j\le d$ we have that 
$M_{k'}^j : \df\Sz\to\C$ is a non-vanishing analytic map. Hence, for any $k\ge d+1$ and $1\le j\le d$, the zero set 
$\mathcal{Z}_{k'}^j:=\big\{u\in\df\Sz\,\big|\,M_{k'}^j(u)=0\big\}$ is nowhere dense in $\df\Sz$. Since, $\Sz$ is a complete metric space we 
then obtain from the Baire category theorem that the set
\[
\mathcal{N}:=\bigcap_{k\ge d+1,1\le j\le d}\big({\df\Sz}\setminus\mathcal{Z}_{k'}^j\big)
\]
is dense in $\df\Sz$. This implies that for any $k\ge d+1$ and $1\le j\le d$, the asymptotic coefficient $a_k^j(t)$
does not vanish in $C(S^{d-1},\R)$ for $t>0$ taken sufficiently small. Since, by Theorem \ref{th:symbol_classes}, 
$a_k^j$ depends analytically on $t\in[0,\tau]$ we conclude that it vanishes only at finitely many $t\in[0,\tau]$.
This completes the proof of Corollary \ref{coro:schwartz_data}.
\end{proof}

\appendix
\section{Auxilliary results}\label{sec:A}
Most of the results in this Appendix are only generalizations to $\gamma+d/p>-1 $ of results in 
\cite{McOwenTopalov2} that assumed $\gamma+d/p>0 $. Rather than repeat the detailed proofs given in \cite{McOwenTopalov2}, 
we will simply describe how to generalization them to the case $\gamma+d/p>-1 $. In one instance, we generalize
a statement from \cite{IKT}. The following lemme follows directly from \eqref{eq:infinity_estimate}.

\begin{Lem}\label{le:A1} 
If $w\in W^{m,p}_\gamma$ with $m>d/p$ and $\gamma+d/p>-1$, then
\begin{equation}\label{eq:A1}
C_1\x\le\big\1 x+w(x)\big\2 \le C_2 \x  \quad\hbox{for} \ x\in \R^d,
\end{equation}
where $C_1,C_2>0$ may be chosen locally uniformly for $w\in W^{m,p}_\gamma$. 
\end{Lem}

\noindent Local uniformity means that for any $w_0\in W^{m,p}_\gamma$ there exists an open neighborhood $U$ of $w_0$ in 
$W^{m,p}_\gamma$ such that the inequality \eqref{eq:A1} holds for any $w\in U$.

Our next result is analogous to Lemma 6.3 in \cite{McOwenTopalov2}. Let $|\dd\varphi(x)|$ denote the sum of the absolute values 
of the elements of the matrix $\dd\varphi(x)$ for $x\in\R^d$.

\begin{Lem}\label{le:A2} 
If $\varphi=\id+w\in {\mathcal D}^{m,p}_\gamma$ with $m>1+d/p$ and $\gamma+d/p>-1$, then
\begin{equation}\label{eq:A2}
\big|\dd\varphi(x)\big|\le C  \quad \hbox{and} \quad 
0<\varepsilon \le\det(\dd\varphi(x))  \quad\hbox{for} \ x\in \R^d,
\end{equation}
where $C$ may be chosen  uniformly for $ \|w\|_{W^{m,p}_\gamma}\le M$
and $\varepsilon$ may be chosen locally uniformly for $w\in W^{m,p}_\gamma$.
\end{Lem}

\begin{proof}[Proof of Lemma \ref{le:A2}] 
It follows from \eqref{eq:infinity_estimate} that for $m>1+d/p$ and $\delta+d/p\ge -1$,
\begin{equation}\label{eq:W->L^infty}
W^{m-1,p}_{\delta+1}\subseteq L^\infty
\end{equation}
is bounded. Take $\varphi=\id+w\in D^{m,p}_\gamma$ with $m>1+d/p$ and $\gamma+d/p>-1$.
Then, the first inequality in \eqref{eq:A2} follows from the boundedness of the inclusions
\[
W^{m,p}_\gamma\stackrel{\dd}{\hookrightarrow}W^{m-q,p}_{\gamma+1}\subseteq L^\infty
\]
and the fact that $\dd\varphi=I+\dd w$. The second inequality in \eqref{eq:A2} follows in a similar way
from \eqref{eq:W->L^infty}, $\dd\varphi=I+\dd w$, and \eqref{eq:infinity_estimate}, since the latter implies 
that $|\dd w|=C\,\|\dd w\|_{W^{m-1,p}_{\delta+1}}/\x^{(\delta+p/d)+1}$ with $(\delta+p/d)+1>0$.
The estimates above a locally uniform for $\varphi\in D^{m,p}_\gamma$.
\end{proof}

The following corollary follows from Lemma \ref{le:A2} and Hadamard-Levi's theorem.
 
\begin{Coro}\label{co:D-open} 
If $\varphi_0=\id+w_0\in {\mathcal D}^{m,p}_\gamma$ where $m>1+d/p$, $\gamma+d/p>-1$, and
$\widetilde w\in W^{m,p}_\gamma$ with $\|\widetilde w\|_{W^{m,p}_\gamma}$ sufficiently small,
then $\varphi_0+\widetilde w\in  {\mathcal D}^{m,p}_\gamma$.
\end{Coro}

In particular, the set of maps ${\mathcal D}^{m,p}_\gamma$ can be identified with an open set in $W^{m,p}_\gamma$. 
Hence, ${\mathcal D}^{m,p}_\gamma$ is a Banach manifold modeled on $W^{m,p}_\gamma$.

For $\varphi\in {\mathcal D}^{m,p}_\gamma$, we know that $\varphi^{-1}$ exists but we need estimates at infinity in order to 
conclude that $\varphi^{-1}\in {\mathcal D}^{m,p}_\gamma$. The following is a first step and is analogous to Lemma 6.4 in 
\cite{McOwenTopalov2}.

\begin{Lem}\label{le:A3} 
If $\varphi=\id+w\in {\mathcal D}^{m,p}_\gamma$ where $m>1+d/p$ and $\gamma+d/p>-1$, then
\[
\big|\dd (\varphi^{-1})(x)\big|\le C  \quad \hbox{and} \quad 
0<\varepsilon \le\det\,\big(\dd(\varphi^{-1})(x)\big)  \quad\hbox{for} \ x\in \R^d,
\]
where $C$ and $\varepsilon$ may be chosen locally uniformly for $ w\in W^{m,p}_\gamma$.
\end{Lem}

\begin{proof}[Proof of Lemma \ref{le:A3}]
The lemma follows from Lemma \ref{le:A2} and the formula
\[
\dd\big(\varphi^{-1}\big)=\big[(\dd\varphi)\circ\varphi^{-1}\big]^{-1}
\]
for the Jacobian matrix of $\varphi^{-1}$.
\end{proof}

Now let us consider compositions of maps as in Theorem \ref{th:group}. We begin with an estimate.

\begin{Lem}\label{le:A4} 
Suppose $m>1+d/p$, $\gamma+d/p>-1$ and $\varphi \in{\mathcal D}^{m,p}_\gamma$. 
Then for every $0\le k\le m$ and $\delta\in\R$, we have
\[
\|w\circ\varphi\|_{W^{k,p}_\delta} \le C\,\|w\|_{W^{k,p}_\delta}\quad\hbox{for all}\ f\in W^{k,p}_\delta,
\]
where $C$ may be taken locally uniformly for $\varphi \in {\mathcal D}^{m,p}_\gamma$.
\end{Lem}

\begin{proof}[Proof of Lemma \ref{le:A4}]
This is proved by induction using Lemma \ref{le:A3} for a change of integration variable when $k=0$ and Proposition 2.2 in 
\cite{McOwenTopalov2} to handle products in the induction step. For details see the proof of Lemma 6.5 in \cite{McOwenTopalov2}.
\end{proof}

\begin{Lem}\label{le:A5} 
Assume $m>1+d/p$, $\gamma+d/p>-1$, $\delta\in\R$, and $f\in C_c^\infty(\R^d)$. 
If $\varphi, \varphi_k\in{\mathcal D}^{m,p}_\gamma$ with $\varphi_k\to\varphi$ in ${\mathcal D}^{m,p}_\gamma$ as $k\to\infty$, 
then $f\circ\varphi_k\to f\circ\varphi$ in $W^{m,p}_\delta(\R^d)$.
\end{Lem}

\begin{proof}[Proof of Lemma \ref{le:A5}]
The only difference from the proof of Lemma 6.6  in \cite{McOwenTopalov2} is that $w_k(x)$ and $w(x)$ in
$\varphi_k(x)=x+w_k(x)$ and $\varphi(x)=x+w(x)$ are of order $O(|x|^{1-\varepsilon})$ for some $\varepsilon>0$
instead of just being bounded. But the rest of the proof in \cite{McOwenTopalov2} can be used here without change.
\end{proof}

Lemmas \ref{le:A4} and \ref{le:A5} can be combined to obtain the continuity of the composition.

\begin{Coro}\label{co:A2}
Suppose $m>1+d/p$, $\gamma+d/p>-1$ and  $\delta\in\R$. Then the composition
$(f,\varphi)\mapsto f\circ\varphi$, $W^{m,p}_\delta\times {\mathcal D}^{m,p}_\gamma\to W^{m,p}_\delta$,
is continuous.
\end{Coro}

\begin{proof}[Proof of Corollary \ref{co:A2}]
The details are the same as in the proof of Corollary 6.1 in \cite{McOwenTopalov2}.
\end{proof}

Next we investigate when the composition is $C^1$.

\begin{Lem}\label{le:A6} 
Assume $m>1+d/p$, $\gamma+d/p>-1$, $\delta\in\R$, and take $\varphi_0\in {\mathcal D}^{m,p}_\gamma$.
For $f\in W^{m+1,p}_\delta$ and $\varphi\in {\mathcal D}^{m,p}_\gamma$ sufficiently close to $\varphi_0$, we have
\[
\|f\circ\varphi-f\circ\varphi_0\|_{W^{m,p}_{\delta+1}}\le C\,\|f\|_{W^{m+1,p}_\delta}\|\varphi-\varphi_0\|_{W^{m,p}_\gamma},
\]
where $C>0$ can be taken uniformly for all $\varphi$ in an open neighborhood of $\varphi_0$.
\end{Lem}

\begin{proof}[Proof of Lemma \ref{le:A6}]
The details are the same as in the proof of Lemma 6.7 in \cite{McOwenTopalov2}.
\end{proof}

\begin{Coro} \label{co:A3} 
Suppose $m>1+d/p$, $\gamma+d/p>-1$ and  $\delta\in\R$. Then the composition
$(f,\varphi)\mapsto f\circ\varphi$, $W^{m+1,p}_\delta\times {\mathcal D}^{m,p}_\gamma\to W^{m,p}_\delta$, is $C^1$.
\end{Coro}

\begin{proof}[Proof of Corollary \ref{co:A3}]
Using the above lemmas, Corollary \ref{co:A3} can be proved following the proof of Proposition 5.1 in \cite{McOwenTopalov2}
(cf. also Appendix B in \cite{XuTopalov}).
\end{proof} 

This completes the proof of Theorem \ref{th:group} (a).
Let us now prove Theorem \ref{th:group} (b). 

\begin{Lem}\label{le:A7} 
If $\varphi\in {\mathcal D}^{m,p}_\gamma$ where $m>1+d/p$ and $\gamma+d/p>-1$, then 
$\varphi^{-1}\in{\mathcal D}^{m,p}_\gamma$.
\end{Lem}

\begin{proof}[Proof of Lemma \ref{le:A7}]
Let $\varphi=\id+w$ and $\varphi^{-1}=\id+u.$ To show $\varphi^{-1}\in {\mathcal D}^{m,p}_\gamma$, 
we need to show $\partial^\alpha u\in L^p_{\gamma+|\alpha|}$ for all $|\alpha|\le m$.
For $\alpha=0 $ we use the change of variables $x=\varphi(y)$ to compute
\[
\begin{aligned}
\int_{\R^d}\x^{\gamma p}|u(x)|^p\,dx&=\int_{\R^d} \x^{\gamma p}|\varphi^{-1}(x)-x|^p\,dx
=\int_{\R^d} \langle\varphi(y)\rangle^{\gamma p}|w(y)|^p\,\det(\dd\varphi(y))\,dy \\
& \le C\,\int \y^{\gamma p}|w(y)|^p\,dy <\infty,
\end{aligned}
\]
where we have used \eqref{eq:A1}, \eqref{eq:A2}, and $w=u\circ\varphi\in W^{m,p}_\gamma$ (cf. Lemma \ref{le:A4}). 
For $1\le |\alpha|\le m$, we can proceed as in (28) in \cite{IKT} to show
\begin{equation}\label{def:F-alpha}
\partial^\alpha (\varphi^{-1}-\id)=F^{(\alpha)}\circ\varphi^{-1},
\end{equation}
where $F^{(\alpha)} : \R^d\to\R^d$ is in $W^{m-|\alpha|,p}_{\gamma+|\alpha|}$. Then, by \eqref{def:F-alpha},
for any $\alpha$ with $0\le|\alpha|\le m$,
\[
\begin{aligned}
\int \x^{(\gamma+|\alpha|)p}\,|\partial^\alpha(\varphi^{-1}(x)-x)|^p\,dx & =
\int \langle\varphi(y)\rangle^{(\gamma+|\alpha|)p}\, |F^{(\alpha)}(y)|^p\,\det(\dd\varphi(y))\,dy \\
& \le C\,\int \y^{(\gamma+|\alpha|)p}\,|F^{(\alpha)}(y)|^p\,dy <\infty.
\end{aligned}
\]
Hence $\varphi^{-1}-\id\in W^{m,p}_\gamma$. This implies that $\varphi\in{\mathcal D}^{m,p}_\gamma$.
\end{proof}

The continuity of the map $\varphi\mapsto\varphi^{-1}$, ${\mathcal D}^{m,p}_\delta\to{\mathcal D}^{m,p}_\delta$,
now follows from Corollary \ref{co:A2} and Theorem 2 in \cite{Mont}.\footnote{As noted in \cite{Mont}, the completeness 
condition in \cite[Theorem 2]{Mont}  can be replaces by local completeness.}
The last statement in Theorem \ref{th:group} (b) can be proved by following 
the proof of Proposition 2.13 in \cite{IKT}.

\begin{Rem}
Note that the regularity assumption $m>1+d/p$ in Lemma \ref{le:A7} above is weaker than the regularity assumption
$m>3+d/p$ in the analogous \cite[Lemma 7.2]{McOwenTopalov2} and \cite[Lemma 2.5]{SultanTopalov}. 
This happens since in \cite{McOwenTopalov2} and \cite{SultanTopalov} one deals with the asymptotic 
expansion of $\varphi^{-1}$, that complicates the situation. Note however, that the results in \cite{McOwenTopalov2} 
and \cite{SultanTopalov} can be extended to the case when $m>1+d/p$ by expanding their proofs.
\end{Rem}

\section{Inverting the Laplace operator}\label{sec:B}
In this Appendix we present, in an extended form, a basic result about the inversion of the Laplace operator in weighted
Sobolev spaces (see Lemma A.3 in \cite{McOwenTopalov3}). Denote by $\mathcal{S}'$ the space of 
tempered distributions in $\R^d$.

\begin{Prop}\label{prop:inverting_the_laplace_operator}
Assume that $d\ge 3$ and $m\ge 0$ with $1<p<\infty$.\footnote{For the case when $d=2$ we refer to 
Proposition 3.3 in \cite{SultanTopalov}.} Then, for any $g\in W^{m.p}_{\delta+2}(\R^d,\R)$ with weight $\delta\in\R$ 
such that $\delta+d/p>0$, $\delta+d/p\notin\Z$, there exists a unique (up to adding a constant term) solution $u$ in 
$\mathcal{S}'\cap L^\infty$ of the Poisson equation
\begin{equation*}
\Delta u=g
\end{equation*}
such that
\begin{equation}\label{eq:Delta^{-1}g}
u:=\Delta^{-1}g=\chi(r)\sum_{d-2\le k<\delta+d/p}\frac{a_k(\theta)}{r^k}+f,\quad f\in W^{m+2,p}_{\delta},
\end{equation}
where $a_k(\theta)$ is an eigenfunction of the Laplace operator $-\Delta_S$ on the unit sphere $S^{d-1}$ in $\R^d$
with eigenvalue $\lambda_{k-d+2}=k(k-d+2)$. If we fix for any $k':=k-d+2\ge 0$ with $d-2\le k<\delta+d/p$ 
an {\em orthonormal} basis 
\[
\big\{h_{k';l}(\theta)\,\big|\, 1\le l\le \nu(k')\big\},\quad \nu(k'):=\dim\mathcal{H}_{k'},
\]
of the eigenspace $\mathcal{H}_{k'}$ of $-\Delta_S$ with eigenvalue $\lambda_{k'}=k'(k'+d-2)$ and expand
\[
a_k(\theta)=\sum_{l=1}^{\nu(k')}\widehat{a}_{k';l}\,h_{k';l}(\theta)
\]
in the Fourier modes, then the Fourier coefficient $\widehat{a}_{k';l}$ equals
\begin{equation}\label{eq:the_integral_formula}
\widehat{a}_{k';l}=-\frac{1}{2k'+d-2}\int_{\R^d}g(x) H_{k';l}(x)\,\dd x,\quad 1\le l\le \nu(k'),
\end{equation}
where $H_{k';l}(x):=h_{k';l}(\theta) r^{k'}$ is the homogeneous harmonic polynomial (of degree $k'\ge 0$) that
corresponds to the Fourier mode $h_{k';l}(\theta)$ and $g H_{k';l}\in L^1(\R^d,\R)$.
If we coordinatize the linear space $\mathcal{S}^{m,p}_\delta$ of functions of the form \eqref{eq:Delta^{-1}g} 
by the Fourier coefficients of $a_k(\theta)$, $d-2\le k<\delta+d/p$, and the reminder $f\in W^{m,p}_\delta$ then the map
\[
\Delta^{-1} : W^{m+2,p}_{\delta+2}\to\mathcal{S}^{m,p}_\delta
\]
is an isomorphism of Banach spaces.
\end{Prop}

\begin{Rem}
Note that the eigenvalues of the Laplace operator on the unit sphere $S^{d-1}$ in $\R^d$ are highly degenerate: 
we have that $\dim\mathcal{H}_0=1$ (these are all constants), $\dim\mathcal{H}_1=d$ (all linear polynomials restricted to 
$S^{d-1}$), and 
\[
\dim\mathcal{H}_{k'}=\binom{d-1+k'}{d-1}-\binom{d-3+k'}{d-3},\quad k'\ge 2,
\]
where the second binomial coefficient above vanishes when $d=2$.
\end{Rem}

\begin{proof}[Proof of Proposition \ref{prop:inverting_the_laplace_operator}]
The fact that there exists a unique solution $u\in\Sz'\cap L^\infty$ of the form \eqref{eq:Delta^{-1}g}
follows from \cite[Lemma A.3(b)]{McOwenTopalov3} (cf. also \cite{McOwen}) and the fact that there is 
a bijective correspondence between the eigenfunctions of the Laplace operator $-\Delta_S$ on the unit sphere 
$S^{d-1}$ in $\R^d$ with eigenvalue $\lambda_{k'}=k'(k'+d-2)$ and the restriction to $S^{d-1}$ of harmonic 
polynomials of degree $k'\ge 0$ (\cite[\S 22.2]{Shubin}.
Let us now prove the integral relation \eqref{eq:the_integral_formula}. To this end, take $d-2\le k<\delta+d/p$, 
$d-2\le n<\delta+d/p$, $1\le l_1\le \nu(k')$, $1\le l_2\le d(n')$ with $n':=n-d+2$, and consider the $k'$-th asymptotic
term $A_{k';l}:=\chi h_{k';l}/r^k$ in \eqref{eq:Delta^{-1}g}. 
Then, it follows from the second Green's identity and the fact that the eigenspaces $\mathcal{H}_{k'}$ and $\mathcal{H}_{n'}$ 
are $L^2$-orthogonal on $S^{d-2}$ for $k'\ne n'$, that
\begin{align}
&\int_{\R^d}\Delta\big(A_{k';l_1}\big) H_{n';l_2}\,\dd x=\lim_{R\to\infty}\int_{S^{d-1}_R}
\Big(\frac{\partial A_{k'l_1}}{\partial r}H_{n';l_2}-A_{k';l_1}\frac{\partial H_{n';l_2}}{\partial r}\Big)
\,d\sigma_R\nonumber\\
&=C(k,n)\lim_{R\to\infty}R^{n'-k'}\int_{S^{d-1}}h_{k';l_1} h_{n';l_2}\,d\sigma_1
=C(k,n)\,\delta_{kn}\label{eq:integral_relation1}
\end{align}
where $C(k,n):=-(k'+n'+d-2)\ne 0$, $S^{d-1}_R$ is the sphere of radius $R$,
and $\delta_{kn}$ is the Kronecker delta. Now, consider the remainder $f\in W^{m+2,p}_\delta$ of the solution $u$ in 
\eqref{eq:Delta^{-1}g}. For any $d-2\le n<\delta+d/p$ and $1\le l\le d(n')$ we have that 
$\Delta\big(f\big) H_{n';l}\in W^{m,p}_{\delta+2-n'}$. This, together with the estimate \eqref{eq:infinity_estimate} implies 
that $\Delta\big(f\big) H_{n';l}=O\big(1/r^{\delta+(d/p)+2-n'}\big)$, and hence $\Delta\big(f\big) H_{n';l}\in L^1(\R^d)$
by the estimate $\delta+(d/p)+2-n'>d$. By arguing in the same way as above, we also have
\begin{align}
&\int_{\R^d}\Delta\big(f\big) H_{n';l}\,\dd x=\lim_{R\to\infty}
\int_{S^{d-1}_R}\Big(\frac{\partial f}{\partial r}\,H_{n';l}-f\,\frac{\partial H_{n';l}}{\partial r}\Big)\,d\sigma_R\nonumber\\
&=\lim_{R\to\infty}\int_{S^{d-1}_R}O\big(1/R^{\delta+(d/p)-n'+1}\big)\,d\sigma_R=
\lim_{R\to\infty}O\big(1/R^{\delta+(d/p)-n}\big)=0,\label{eq:integral_relation2}
\end{align}
where we used that $n<\delta+d/p$ and the estimate \eqref{eq:infinity_estimate} on the decay of 
$f$ at infinity. Finally, the integral formula \eqref{eq:the_integral_formula} follows from \eqref{eq:integral_relation1},
\eqref{eq:integral_relation2}, and \eqref{eq:Delta^{-1}g}.
\end{proof}


\section{Asymptotic spaces $\A^{m,p}_{n,N;\ell}$}\label{sec:C}
In this Section we will discuss the asymptotic spaces with log terms that are used in the proof of Theorem \ref{th:main2}.
For integers $m>d/p$, $0\leq n\leq N$, and $\ell\geq -n$, let $\A^{m,p}_{n,N;\ell}$ denote functions $u$ on $\R^d$ 
of the form
\begin{subequations}
\begin{equation}\label{eq:asymptotic_expansion_n,N}
\chi(r)\left(\frac{a_n^0(\theta)+\cdots +a_n^{n+\ell}(\theta)(\log r)^{n+\ell}}{r^n}+\cdots+
\frac{a_N^0(\theta)+\cdots+a_N^{N+\ell}(\theta)(\log r)^{N+\ell}}{r^N}\right)+f(x),
\end{equation}
where $a_k^j\in H^{m+1+N-k,p}(\s^{d-1},\R^d)$ for $0\leq j\leq k+\ell$, $0\leq n\leq k\leq N$, and $f\in W_{\gamma_N}^{m,p}$
with $N\le\gamma_N+d/p<N+1$ so, by \eqref{eq:infinity_growth}, $f(x)=o\big(r^{-N}\big)$ as $r\to\infty$. 
This is a Banach space space with norm
\begin{equation}\label{eq:norm-A_n,N}
\|u\|_{\A^{m,p}_{n,N;\ell}}=\sum_{0\leq j\leq k+\ell, \, n\leq k\leq N} \|a_k^j\|_{H^{m+1+N-k,p}}
\ + \ \|f\|_{W^{m,p}_{\gamma_N}}.
\end{equation}
\end{subequations}
Note that $W_{\gamma_N}^{m,p}$ is a closed subspace of $\A^{m,p}_{n,N;\ell}$.
It is easy to confirm that the following inclusions are bounded for $N\le\gamma_N+d/p<N=1$:
\begin{subequations}
\begin{equation}\label{A-prop1}
\A^{m,p}_{n_1,N_1;\ell_1} \subseteq\A^{m,p}_{n,N;\ell} \quad \hbox{if}\ n_1\geq n,\ N_1\geq N,\ \ell\geq\ell_1\geq -n,
\end{equation}
\begin{equation}\label{A-prop2}
\partial_{x_J} : \A^{m,p}_{n,N;\ell} \to \A^{m-1,p}_{n+1,N+1;\ell-1} \quad \hbox{if}\ m>1+d/p,\quad 1\le j\le d.
\end{equation}
Moreover, if $n=n_1+n_2$, $\ell=\ell_1+\ell_2$, and $N<\min(N_1+n_2,N_2+n_1)$, then pointwise 
multiplication $(u,v)\mapsto uv$,
\begin{equation}\label{A-prop3}
\A^{m,p}_{n_1,N_1;\ell_1} \times \A^{m,p}_{n_2,N_2;\ell_2} \to 
\A^{m,p}_{n,N;\ell},
\end{equation}
is a bounded bilinear map.
When $\ell_i=-n_i$ there are no log terms in the leading asymptotic and we have the sharper version with 
$N=\min(N_1+n_2,N_2+n_1)$, $n=n_1+n_1$:
\begin{equation}\label{A-prop4}
\A^{m,p}_{n_1,N_1;-n_1} \times \A^{m,p}_{n_2,N_2;-n_2} \to 
\A^{m,p}_{n,N;\ell}.
\end{equation}
These may be combined to conclude that
\begin{equation}\label{A-prop5}
\A^{m,p}_{n,N;\ell} \ \hbox{is a Banach algebra}
\end{equation}
\end{subequations}
in the case when $n\ge 1$, or when $\ell=-n$.
For more details, see Appendix B in \cite{McOwenTopalov2} (for the case when $N<\gamma_N<N+1$).

Analogous to \eqref{def:D}, for $m>1+d/p$ we introduce asymptotic diffeomorphisms
\begin{equation}\label{def:AD}
{\A\mathcal D}^{m,p}_{n,N;\ell}:=\big\{\varphi : \R^d\to\R^d\,\big|\,\varphi=\id+w,\,w\in \A^{m,p}_{n,N;\ell}\,
\text{ and}\,\det(\dd\varphi)>0\big\}.
\end{equation}
Analogous to Theorem \ref{Smoothness-Euler} above,
Theorem 6.1 in \cite{McOwenTopalov3} shows for $m>3+d/p$ and $0\leq n\leq \min(d+1,N)$ that the Euler vector field 
\begin{equation}\label{Euler-AD}
{\mathcal E}: {\A\mathcal D}^{m,p}_{n,N;0}  \times {\A}^{m,p}_{n,N;0} \to {\A\mathcal D}^{m,p}_{n,N;0} \times {\A}^{m,p}_{n,N;0}
\end{equation}
is smooth. Theorem 1.1 in \cite{McOwenTopalov3} uses this vector field as we did in Section \ref{sec:ProofTheorem1} to show
the existence of a unique solution as in \eqref{eq:u->asymptotic_space}.
However, all results in  \cite{McOwenTopalov3}  are under the assumption
$N<\gamma_N+d/p<N+1$.

We now show that this solvability of the Euler equations in $\A^{m,p}_{n,N;0}$ also holds for $N\le\gamma_N+d/p<N+1$, at least when $n=1$.

\begin{Prop}\label{pr:Euler-n=1}
Assume $m>3+d/p$, $N\geq 1$, and $N\le\gamma_N<N+1$. Then, for any given $\rho>0$ there exists $\tau>0$ such that
for any $u_0\in\ddf \A^{m,p}_{1,N;0}$ with $\|u_0\|_{\A^{m,p}_{1,N;0}}<\rho$  there exists a unique solution of the Euler equation
\begin{equation}\label{eq:u->asymptotic_space*}
u\in C\big([0,\tau],\ddf\A^{m,p}_{1,N;0}\big)\cap C^1\big([0,\tau],\ddf\A^{m-1,p}_{1,N;0}\big),
\end{equation}
that depends continuously on the initial data $u_0$.
\end{Prop}

\begin{proof}[Proof of Proposition \ref{pr:Euler-n=1}]
As previously stated, this was proved in \cite{McOwenTopalov3} when $N<\gamma_N+d/p<N+1$, so we need only consider
the case when $\gamma_N+d/p=N$. To do this, let us change notation within this proof and denote by $\mathfrak{A}^{m,p}_{n,N;0}$
the asymptotic space $\A^{m,p}_{n,N;\ell}$ with $\gamma_N+d/p=N$; we reserve the notation $\A^{m,p}_{n,N;0}$ 
for an asymptotic space with a remainder $f\in W^{m,p}_{\tilde\gamma_N}$ with $N<\tilde\gamma_N<N+1$. In particular, 
referring to \eqref{A-prop1}, we have the following bounded inclusions for any $\ell\geq -n$:
\begin{equation}\label{eq:A-inclusions}
\A^{m,p}_{n,N;\ell}\subseteq\mathfrak{A}^{m,p}_{n,N;\ell}\subseteq\A^{m,p}_{n,N-1;\ell}.
\end{equation}
We will need to take the square of elements in $\mathfrak{A}^{m,p}_{n,N;\ell}$. For $n+\ell>0$, we know from \eqref{A-prop3} and  
\eqref{A-prop4} that the log terms in the leading asymptotic prevent us from keeping the same order of decay; but decreasing this 
order of decay by one enables us to also embed in a space with $0<\gamma_0+d/p<1$:
\begin{equation}\label{eq:u->u^2}
u\mapsto u^2 \quad\hbox{is smooth} \quad \mathfrak{A}^{m,p}_{n,N;\ell} \to \A^{m,p}_{2n,N+n-1;2\ell}.
\end{equation}
Using \eqref{A-prop2}, for $u\in \mathfrak{A}^{m,p}_{1,N;0}$ we have $\dd u\in \mathfrak{A}^{m-1,p}_{2,N+1;-1}$ so 
\eqref{eq:u->u^2} implies that $Q(u)\equiv\tr\big(\dd u\big)^2$ defines a bounded quadratic polynomial map 
$Q :\mathfrak{A}^{m,p}_{1,N;0}\to\A^{m-1,p}_{4,N+2;-2}\subseteq\A^{m-1,p}_{2,N+2;-2}$. Consequently, the maps
\begin{equation}\label{eq:Q-trick}
Q : \mathfrak{A}^{m,p}_{1,N;0}\to\A^{m-1,p}_{2,N+2;-2}\quad\text{\rm and}\quad
\nabla\circ Q : \mathfrak{A}^{m,p}_{1,N;0}\to\A^{m-2,p}_{3,N+3;-3}
\end{equation}
are smooth. We can now apply Proposition 3.1 and (17b) in \cite{McOwenTopalov3}  to conclude  that the linear map 
$\Delta^{-1}: \A^{m-2,p}_{3,N+3;-3}\to \A^{m,p}_{1,N+1;0}$ is bounded and injective. 
Combined with \eqref{eq:Q-trick} and the embeddings $\A^{m,p}_{1,N+1;0}\subseteq\A^{m,p}_{1,N;0} \subseteq\mathfrak{A}^{m,p}_{1,N;0}$
we see that
\[
\Delta^{-1}\circ\nabla\circ Q : \mathfrak{A}^{m,p}_{1,N;0}\to \mathfrak{A}^{m,p}_{1,N;0}
\]
is smooth.  The arguments in \cite[Section 4, 5, and 6]{McOwenTopalov3} (cf. Lemma 5.2, Proposition 5.1, Lemma 6.1, and Theorem 6.1
in \cite{McOwenTopalov3}) then show that the associated conjugate map
\[
(\varphi,f)\mapsto\big(R_\varphi\circ\Delta^{-1}\circ R_{\varphi^{-1}}\big)\circ\big(R_{\varphi}\circ\nabla\circ Q\circ R_{\varphi^{-1}}\big)(f),
\quad\mathfrak{A}{\mathcal D}^{m,p}_{1,N;0}\times\mathfrak{A}^{m,p}_{1,N;0}\to
\mathfrak{A}{\mathcal D}^{m,p}_{1,N;0}\times\mathfrak{A}^{m,p}_{1,N;0},
\]
is smooth. This implies that the Euler vector field $\mathcal{E}$ is smooth as a map
\begin{equation}\label{Vectorfield-AD}
\mathcal{E} : {\mathcal AD}^{m,p}_{1,N;0}\times\A^{m,p}_{1,N;0}\to{\mathcal AD}^{m,p}_{1,N;0}\times\A^{m,p}_{1,N;0}.
\end{equation}
The arguments in the proof of Theorem 1.1 above then complete the proof  of the proposition (cf. also Section 7 in \cite{McOwenTopalov3}).
\end{proof}


\section{Global existence in the case when $d=2$.}\label{sec:D}
In this Section we generalize Theorem 1.1 in \cite{SultanTopalov} and prove that for $d=2$ the solution of the Euler equation \eqref{eq:euler} 
has a unique global in time solution in the asymptotic space ${\mathcal Z}^{m,p}_N$ with weight $\gamma_N$ such that $\gamma_N+d/p>0$ is 
integer.\footnote{In particular, we obtain an alternative proof of Proposition \ref{pr:Euler-n=1} in the case $d=2$.}
For a given $a\in\R$ denote by $\floor{a}$ the integer part of $a$. 
We will follow the notation introduced in \cite[Section 2]{SultanTopalov}. 

For a given $1<p<\infty$, $m>2/p$, and $\delta+2/p>0$ we set $N:=\floor{\delta+2/p}$, $\gamma_N:=\delta$, and consider the space 
of complex valued functions of $z\in\C$,
\begin{equation}\label{eq:Z}
\mathcal{Z}^{m,p}_{n,N}:=\Big\{\chi\sum_{n\le k+l\le N}\frac{a_{kl}}{z^k{\bar z}^l}+f\,\Big|\,f\in W^{m,p}_{\gamma_N}, a_{kl}\in\C\Big\},
\end{equation}
where $0\le n\le N+1$ and where we omit the summation term if $n=N+1$ and set $\mathcal{Z}^{m,p}_{n,N}\equiv W^{m,p}_{\gamma_N}$.
We also set $\mathcal{Z}^{m,p}_N\equiv\mathcal{Z}^{m,p}_{0,N}$.
The space \eqref{eq:Z} is a closed subspace in the asymptotic space $\A^{m,p}_N$ of vector fields on $\R^2$
that  satisfies Proposition 2.1 and 2.2 in \cite{SultanTopalov}. Note however, that Proposition 3.3 and Theorem 3.2 
in \cite{SultanTopalov} does {\em not} hold for integer $\delta+2/p$. As a consequence, the proof of the
global well-posedness of the Euler equation for $d=2$ in \cite[Section 5]{SultanTopalov} does {\em not} apply for integer values of $\delta+2/p$.
Following  \cite[Section 2]{SultanTopalov} we denote the group of diffeomorphisms of $\R^2$ modeled on $\mathcal{Z}^{m,p}_N$
by ${\mathcal Z}D^{m,p}_N$. First, we prove the following lemma.

\begin{Lem}\label{lem:conjugation}
Take $m>3+2/p$, a non-integer $\delta+2/p>0$, and let $\hat\delta$ be the lowest integer $\hat\delta>\delta$ such that $\hat\delta+2/p\in\Z$.
Then, for a given volume preserving $\varphi\in{\mathcal Z}D^{m,p}_M$ (with $\gamma_M:=\delta$, $M:=\floor{\delta+2/p}$ ) and 
$u_0\in\mathcal{Z}^{m,p}_N$ (with $\gamma_N:=\hat\delta$, $N:={\hat\delta}+2/p$) 
we have that
\begin{equation}\label{eq:conjugation}
\big(R_\varphi\circ\partial_z^{-1}\circ R_{\varphi^{-1}}\big)(\partial_z u_0)=u_0+\mathcal{R}(\varphi,u_0),\quad
\mathcal{R}(\varphi,u_0)\in\mathcal{Z}^{m,p}_{1,M+1},
\end{equation}
where the map $\mathcal{R} : {\mathcal Z}D^{m,p}_M\times{\mathcal Z}^{m,p}_N\to{\mathcal Z}^{m,p}_{M+1}$
is analytic and $\partial_z$ denotes the Cauchy operator 
$\partial_z : \mathcal{Z}^{m,p}_{1,M}\to\widetilde{\mathcal{Z}}^{m-1,p}_{M+1}$.\footnote{Theorem 3.2 in \cite{SultanTopalov}
states that this map is an isomorphism; see (16) in \cite{SultanTopalov} for the definition of the space $\widetilde{\mathcal{Z}}^{m-1,p}_{M+1}$.}
\end{Lem}

\begin{proof}[Proof of Lemma \ref{lem:conjugation}]
Since $u_0\in\mathcal{Z}^{m,p}_{N}$ (with $\gamma_N=\hat\delta$) we have that $\partial_z u_0\in\widetilde{\mathcal{Z}}^{m-1,p}_{1,M+1}$ and, 
by Proposition 3.4 in \cite{SultanTopalov}, $\big(R_\varphi\circ\partial_z^{-1}\circ R_{\varphi^{-1}}\big)(\partial_z u_0)$ is well defined 
and belongs to $\mathcal{Z}^{m,p}_{1,M}$.
By setting $w:=\big(R_\varphi\circ\partial_z^{-1}\circ R_{\varphi^{-1}}\big)(\partial_z u_0)$ we then obtain from Lemma 2.4
in \cite{SultanTopalov} that
\[
\big(R_\varphi\circ\partial_z\circ R_{\varphi^{-1}}\big)(w)=\partial_z u_0.
\]
This, together with formula (54) in \cite{SultanTopalov} and the fact that $\varphi=\id_\C+u\in{\mathcal Z}D^{m,p}_M$ is volume preserving,
then implies that $\partial_z w+(\partial_z w)(\partial_{\bar z}{\bar u}_0)-(\partial_{\bar z}w)(\partial_z{\bar u}_0)=\partial_z u_0$, or equivalently,
\begin{equation}\label{eq:conjugation'}
w=u_0+\partial_z^{-1}\big[(\partial_{\bar z}w)(\partial_z{\bar u}_0)-(\partial_z w)(\partial_{\bar z}{\bar u}_0)\big].
\end{equation}
Since, by Lemma 3.5 in \cite{SultanTopalov}, 
$w\equiv\big(R_\varphi\circ\partial_z^{-1}\circ R_{\varphi^{-1}}\big)(\partial_z u_0)\in\mathcal{Z}^{m,p}_{1,M}$ depends analytically on
$(\varphi,u_0)\in{\mathcal Z}D^{m,p}_M\times{\mathcal Z}D^{m,p}_N$, we obtain from Proposition 2.2 in \cite{SultanTopalov} that
\[
(\partial_{\bar z}w)(\partial_z{\bar u}_0)-(\partial_z w)(\partial_{\bar z}{\bar u}_0)\in\widetilde{\mathcal{Z}}^{m-1,p}_{M+2}
\]
and depends analytically on $(\varphi,u_0)\in{\mathcal Z}D^{m,p}_M\times{\mathcal Z}^{m,p}_N$. By combining this with
Theorem 3.2 in \cite{SultanTopalov} we then see that
\[
\mathcal{R}(\varphi,u_0):=\partial_z^{-1}\big[(\partial_{\bar z}w)(\partial_z{\bar u}_0)-(\partial_z w)(\partial_{\bar z}{\bar u}_0)\big]
\in\mathcal{Z}^{m,p}_{1,M+1}
\]
and depends analytically on $(\varphi,u_0)\in{\mathcal Z}D^{m,p}_M\times{\mathcal Z}^{m,p}_N$.
This completes the proof of the lemma.
\end{proof}

Now, we are ready to prove

\begin{Prop}\label{prop:global_existence_d=2}
Assume that $m>3+2/p$, $\delta+2/p>0$ is an integer, and $d=2$. Then, for any $u_0\in\df{\mathcal Z}^{m,p}_N$ 
(with $\gamma_N:=\delta$ and $N:=\delta+2/p$) the Euler equation \eqref{eq:euler} has a unique global in time solution 
$u\in C\big([0,\infty),\df{\mathcal Z}^{m,p}_N\big)\cap C^1\big([0,\infty),\df{\mathcal Z}^{m-1,p}_N\big)$
that depends continuously on the initial data (cf. \cite[Theorem 1.1]{SultanTopalov} for the case when $\gamma_N+2/p$ is not 
integer).
\end{Prop}

\begin{proof}[Proof of Proposition \ref{prop:global_existence_d=2}]
Assume that $\delta+2/p>0$ is an integer and choose $\delta^-\in\R$ such that $0<\delta-\delta^-<1$ and $\delta^-+2/p>0$.
Take $u_0\in{\mathcal Z}^{m,p}_N$ (with $\gamma_N=\delta$). Since ${\mathcal Z}^{m,p}_N$ is a subspace in ${\mathcal Z}^{m,p}_M$ 
(with $\gamma_M:=\delta^-$ and $M:=\floor{\delta^-+2/p}=N-1$) and since $\delta^-+2/p$ is not integer, we conclude from 
\cite[Theorem 1.1]{SultanTopalov}  that there exists a unique solution 
of the Euler equation
\[
u\in C\big([0,\infty),\mathcal{Z}^{m,p}_M\big)\cap C^1\big([0,\infty),\mathcal{Z}^{m-1,p}_M\big)
\]
that depends continuously on the initial data $u_0\in\mathcal{Z}^{m,p}_N$.
By \cite[Proposition 4.2]{SultanTopalov}, $\varphi\in C^1\big([0,\infty),\mathcal{Z}D^{m,p}_M\big)$
where $\dt\varphi=u\circ\varphi$, $\varphi|_{t=0}=u_0$, and it depends continuously on the initial data
$u_0\in\mathcal{Z}^{m,p}_N$. The preservation of vorticity (cf. formula (76) in \cite{SultanTopalov}) and 
Lemma \ref{lem:conjugation} then imply that
\begin{align}
u&=\partial_z^{-1}\big((\partial_z u_0)\circ\varphi^{-1}\big)=
R_{\varphi^{-1}}\circ\big(R_\varphi\circ\partial_z^{-1}\circ R_{\varphi^{-1}}\big)(\partial_z u_0)\nonumber\\
&=u_0\circ\varphi^{-1}+\mathcal{R}(\varphi,u_0)\circ\varphi^{-1}
\end{align}
where $\mathcal{R}(\varphi,u_0)\in\mathcal{Z}^{m,p}_{1,M+1}$ and it depends analytically on 
$(\varphi,u_0)\in{\mathcal Z}D^{m,p}_M\times{\mathcal Z}^{m,p}_N$. 
Since $\gamma_M+1>\delta$ we have that $\mathcal{Z}^{m,p}_{M+1}\subseteq\mathcal{Z}^{m,p}_N$. 
By Proposition 2.3 and  Proposition 2.4 in \cite{SultanTopalov} we then obtain that
\[
u\in C\big([0,\infty),\mathcal{Z}^{m,p}_N\big)\cap C^1\big([0,\infty),\mathcal{Z}^{m-1,p}_N\big)
\]
and it depends continuously on the initial data $u_0\in\mathcal{Z}^{m,p}_N$.
This completes the proof of the proposition.
\end{proof}


\end{document}